 \documentclass[12pt]{amsart}
\usepackage[T1]{fontenc}
\usepackage{amssymb}
\usepackage{a4wide}
\usepackage{graphicx}
\usepackage{verbatim}
\usepackage{amsmath}
\usepackage{version}

\newcommand{\be}{\begin{equation}}
\newcommand{\ee}{\end{equation}}

 \theoremstyle{plain}    
 \newtheorem{thm}{Theorem}[section]
 \numberwithin{equation}{section} 
 \numberwithin{figure}{section} 
 \theoremstyle{plain}
 \newtheorem{prop}[thm]{Proposition} 
 \theoremstyle{plain}    
 \newtheorem{lem}[thm]{Lemma} 
 \theoremstyle{plain}    
 \theoremstyle{plain}    

 \theoremstyle{definition}
 \newtheorem{rem}[thm]{Remark} 
 \theoremstyle{definition}

\begin{document}
\newcommand{\nwc}{\newcommand}
\nwc{\nwt}{\newtheorem}
\nwt{coro}{Corollary}
\nwt{ex}{Example}


\nwc{\mf}{\mathbf} 
\nwc{\blds}{\boldsymbol} 
\nwc{\ml}{\mathcal} 


\nwc{\lam}{\lambda}
\nwc{\om}{\omega}
\nwc{\del}{\delta}
\nwc{\Del}{\Delta}
\nwc{\Lam}{\Lambda}
\nwc{\elll}{\ell}

\nwc{\IA}{\mathbb{A}} 
\nwc{\IB}{\mathbb{B}} 
\nwc{\IC}{\mathbb{C}} 
\nwc{\ID}{\mathbb{D}} 
\nwc{\IE}{\mathbb{E}} 
\nwc{\IF}{\mathbb{F}} 
\nwc{\IG}{\mathbb{G}} 
\nwc{\IH}{\mathbb{H}} 
\nwc{\IN}{\mathbb{N}} 
\nwc{\IP}{\mathbb{P}} 
\nwc{\IQ}{\mathbb{Q}} 
\nwc{\IR}{\mathbb{R}} 
\nwc{\IS}{\mathbb{S}} 
\nwc{\IT}{\mathbb{T}} 
\nwc{\IZ}{\mathbb{Z}} 
\def\bbbone{{\mathchoice {1\mskip-4mu {\rm{l}}} {1\mskip-4mu {\rm{l}}}
{ 1\mskip-4.5mu {\rm{l}}} { 1\mskip-5mu {\rm{l}}}}}
\def\bbleft{{\mathchoice {[\mskip-3mu {[}} {[\mskip-3mu {[}}{[\mskip-4mu {[}}{[\mskip-5mu {[}}}}
\def\bbright{{\mathchoice {]\mskip-3mu {]}} {]\mskip-3mu {]}}{]\mskip-4mu {]}}{]\mskip-5mu {]}}}}
\nwc{\setK}{\bbleft 1,K \bbright}
\nwc{\setN}{\bbleft 1,\cN \bbright}


\nwc{\va}{{\bf a}}
\nwc{\vb}{{\bf b}}
\nwc{\vc}{{\bf c}}
\nwc{\vd}{{\bf d}}
\nwc{\ve}{{\bf e}}
\nwc{\vf}{{\bf f}}
\nwc{\vg}{{\bf g}}
\nwc{\vh}{{\bf h}}
\nwc{\vi}{{\bf i}}
\nwc{\vI}{{\bf I}}
\nwc{\vj}{{\bf j}}
\nwc{\vk}{{\bf k}}
\nwc{\vl}{{\bf l}}
\nwc{\vm}{{\bf m}}
\nwc{\vM}{{\bf M}}
\nwc{\vn}{{\bf n}}
\nwc{\vo}{{\it o}}
\nwc{\vp}{{\bf p}}
\nwc{\vq}{{\bf q}}
\nwc{\vr}{{\bf r}}
\nwc{\vs}{{\bf s}}
\nwc{\vt}{{\bf t}}
\nwc{\vu}{{\bf u}}
\nwc{\vv}{{\bf v}}
\nwc{\vw}{{\bf w}}
\nwc{\vx}{{\bf x}}
\nwc{\vy}{{\bf y}}
\nwc{\vz}{{\bf z}}
\nwc{\bal}{\blds{\alpha}}
\nwc{\bep}{\blds{\epsilon}}
\nwc{\bom}{\blds{\omega}}
\nwc{\barbep}{\overline{\blds{\epsilon}}}
\nwc{\bnu}{\blds{\nu}}
\nwc{\bmu}{\blds{\mu}}



\nwc{\bk}{\blds{k}}
\nwc{\bm}{\blds{m}}
\nwc{\bM}{\blds{M}}
\nwc{\bp}{\blds{p}}
\nwc{\bq}{\blds{q}}
\nwc{\bn}{\blds{n}}
\nwc{\bv}{\blds{v}}
\nwc{\bw}{\blds{w}}
\nwc{\bx}{\blds{x}}
\nwc{\bxi}{\blds{\xi}}
\nwc{\by}{\blds{y}}
\nwc{\bz}{\blds{z}}


\nwc{\cA}{\ml{A}}
\nwc{\cB}{\ml{B}}
\nwc{\cC}{\ml{C}}
\nwc{\cD}{\ml{D}}
\nwc{\cE}{\ml{E}}
\nwc{\cF}{\ml{F}}
\nwc{\cG}{\ml{G}}
\nwc{\cH}{\ml{H}}
\nwc{\cI}{\ml{I}}
\nwc{\cJ}{\ml{J}}
\nwc{\cK}{\ml{K}}
\nwc{\cL}{\ml{L}}
\nwc{\cM}{\ml{M}}
\nwc{\cN}{\ml{N}}
\nwc{\cO}{\ml{O}}
\nwc{\cP}{\ml{P}}
\nwc{\cQ}{\ml{Q}}
\nwc{\cR}{\ml{R}}
\nwc{\cS}{\ml{S}}
\nwc{\cT}{\ml{T}}
\nwc{\cU}{\ml{U}}
\nwc{\cV}{\ml{V}}
\nwc{\cW}{\ml{W}}
\nwc{\cX}{\ml{X}}
\nwc{\cY}{\ml{Y}}
\nwc{\cZ}{\ml{Z}}


\nwc{\tA}{\widetilde{A}}
\nwc{\tB}{\widetilde{B}}
\nwc{\tE}{E^{\vareps}}
\nwc{\tk}{\tilde k}
\nwc{\tN}{\tilde N}
\nwc{\tP}{\widetilde{P}}
\nwc{\tQ}{\widetilde{Q}}
\nwc{\tR}{\widetilde{R}}
\nwc{\tV}{\widetilde{V}}
\nwc{\tW}{\widetilde{W}}
\nwc{\ty}{\tilde y}
\nwc{\teta}{\tilde \eta}
\nwc{\tdelta}{\tilde \delta}
\nwc{\tlambda}{\tilde \lambda}
\nwc{\ttheta}{\tilde \theta}
\nwc{\tvartheta}{\tilde \vartheta}
\nwc{\tPhi}{\widetilde \Phi}
\nwc{\tpsi}{\tilde \psi}

\nwc{\fA}{\mathfrak{A}}
\nwc{\fN}{\mathfrak{N}}
\nwc{\fM}{\mathfrak{M}}
\nwc{\fP}{\frak{P}}
\nwc{\fG}{\frak{G}}
\nwc{\fK}{\frak{K}}
\nwc{\fH}{\frak{H}}

\nwc{\To}{\longrightarrow} 

\nwc{\ad}{\rm ad}
\nwc{\eps}{\epsilon}
\nwc{\ep}{\epsilon}
\nwc{\vareps}{\varepsilon}

\def\ep{\epsilon}
\def\al{\alpha}
\def\tr{{\rm tr}}
\def\Tr{{\rm Tr}}
\def\i{{\rm i}}
\def\mi{{\rm i}}
\def\e{{\rm e}}
\def\sq2{\sqrt{2}}
\def\sqn{\sqrt{N}}
\def\vol{\mathrm{vol}}
\def\defi{\stackrel{\rm def}{=}}
\def\slog{\stackrel{\rm log}{\simeq}}
\def\llog{\stackrel{\rm log}{\lesssim}}
\def\glog{\stackrel{\rm log}{\gtrsim}}
\def\t2{{\mathbb T}^2}
\def\s2{{\mathbb S}^2}
\def\hn{\mathcal{H}_{N}}
\def\shbar{\sqrt{\hbar}}
\def\A{\mathcal{A}}
\def\N{\mathbb{N}}
\def\T{\mathbb{T}}
\def\R{\mathbb{R}}
\def\Z{\mathbb{Z}}
\def\C{\mathbb{C}}
\def\O{\mathcal{O}}
\def\Sp{\mathcal{S}_+}
\def\Lap{\triangle}
\nwc{\lap}{\bigtriangleup}
\nwc{\rest}{\restriction}
\nwc{\Diff}{\operatorname{Diff}}
\nwc{\diff}{\operatorname{diff}}
\nwc{\diam}{\operatorname{diam}}
\nwc{\Res}{\operatorname{Res}}
\nwc{\Spec}{\operatorname{Spec}}
\nwc{\Vol}{\operatorname{Vol}}
\nwc{\Op}{\operatorname{Op}}
\nwc{\supp}{\operatorname{supp}}
\nwc{\Span}{\operatorname{span}}

\nwc{\dia}{\varepsilon}
\nwc{\cut}{f}
\nwc{\qm}{u_\hbar}

\def\hto0{\xrightarrow{\hbar\to 0}}
\def\htoo{\stackrel{h\to 0}{\longrightarrow}}
\def\rto0{\xrightarrow{r\to 0}}
\def\rtoo{\stackrel{r\to 0}{\longrightarrow}}
\def\ntoinf{\xrightarrow{n\to +\infty}}
\def\TS{\mathcal{T}^*S}
\def\TY{\mathcal{T}^*Y}

\providecommand{\abs}[1]{\lvert#1\rvert}
\providecommand{\norm}[1]{\lVert#1\rVert}
\providecommand{\set}[1]{\left\{#1\right\}}

\nwc{\la}{\langle}
\nwc{\ra}{\rangle}
\nwc{\lp}{\left(}
\nwc{\rp}{\right)}

\nwc{\bequ}{\begin{equation}}
\nwc{\ben}{\begin{equation*}}
\nwc{\bea}{\begin{eqnarray}}
\nwc{\bean}{\begin{eqnarray*}}
\nwc{\bit}{\begin{itemize}}
\nwc{\bver}{\begin{verbatim}}

%\nwc{\eal}{\end{align}}
\nwc{\eequ}{\end{equation}}
\nwc{\een}{\end{equation*}}
\nwc{\eea}{\end{eqnarray}}
\nwc{\eean}{\end{eqnarray*}}
\nwc{\eit}{\end{itemize}}
\nwc{\ever}{\end{verbatim}}

\newcommand{\defeq}{\stackrel{\rm{def}}{=}}
\newcommand{\Lim}{\mathop{\longrightarrow}\limits}

\title[Damped wave equation]
{Spectral deviations for the damped wave equation}

\author[N. Anantharaman]{Nalini Anantharaman}
\address{CMLS, \'Ecole Polytechnique, 91128 Palaiseau Cedex, France}
\email{nalini@math.polytechnique.fr}
\begin{abstract} We prove a Weyl-type fractal upper bound for the spectrum of the damped wave equation, on a negatively curved compact manifold. It is known that most of the eigenvalues have an imaginary part close
to the average of the damping function.
We count the number of eigenvalues in a given horizontal strip deviating from this typical behaviour; the exponent that appears naturally is the `entropy' that gives the deviation rate from the Birkhoff ergodic theorem for the geodesic flow. A Weyl-type lower bound is still far from reach; but in the particular case of arithmetic surfaces, and for a strong enough damping, we can use the trace formula to prove a result going in this direction.
\end{abstract}
 
\maketitle

\section{Results and questions about the spectrum of the damped wave equation}
Let $(M, g)$ be a smooth compact Riemannian manifold without boundary.
Let $a$ be a $C^\infty$ {\em real valued} function on $M$.
We study the ``damped\footnote{The term ``damped'' applies to the case when $a\geq 0$, that is to say, when the energy is decreasing. In this paper the sign of $a$ will be of no importance.} wave equation'',
\begin{equation}
 \label{e:DWE}
\left(\partial^2_t-\Lap+2a(x)\partial_t\right)v=0
\end{equation}
for $t\in \IR$ and $x\in M$. We shall be interested in the stationary solutions, that is, solutions of the form
$v(t, x)=e^{it\tau}u(x)$ for some $\tau\in \IC$. This means that $u$ must satisfy
\begin{equation}
 \label{e:SDWE}(-\lap-\tau^2+2ia\tau)u=0.
 \end{equation}
Equivalently, $\tau$ is an eigenvalue of the operator
$$\left( \begin{array}{cc} 

0 & I \\
-\lap & 2ia \\
 
 \end{array}  \right)$$
 acting on $H^1(M)\times L^2(M)$. For $a=0$ this operator is the wave operator, an anti-selfadjoint operator;
 but for $a\not=0$ the operator is not normal anymore. It is known that its spectrum is discrete
 and consists of a sequence $(\tau_n)$ with $\Im m(\tau_n)$ bounded and $|\Re e(\tau_n)|\To +\infty$ (see Section \ref{s:spec}). One sees easily that $\Im m(\tau_n)\in [2\min(\inf a, 0), 2\max(\sup a, 0)]$ if $\Re e(\tau_n)=0$, and $\Im m(\tau_n)\in [\inf a, \sup a]$ if $\Re e(\tau_n)\not=0$ \cite{Sj00}. One can also note, obviously, that $-\bar \tau_n$ is an eigenvalue if $\tau_n$ is~: the spectrum is symmetric with respect to the imaginary axis.
 
 \subsection{Background}
 Motivated by questions from control theory, Lebeau \cite{Leb93} was interested in the so-called ``stabilization problem''~: define
 $$\rho=\sup\left\{\beta, \,\exists\, C,\, E(u_t)\leq Ce^{-\beta t}E(u_0)\mbox{ for every solution } u
 \mbox{ of  \eqref{e:DWE}}\right\}$$
 where the energy functional is $E(u)=\int_M |\nabla u|^2$. The stabilization problem consists in finding
 some damping function $a$ (necessarily nonnegative) such that $\rho>0$. Lebeau identified
 $$\rho=2\min\left\{D(0), C(\infty)\right\}$$
 where $D(0)$ is a sort of ``spectral gap''~:
$$D(0)=\inf \left\{\Im m(\tau_n), \tau_n\not= 0\right\}$$
 and $C(\infty)$ is defined
 in terms of the Birkhoff averages of $a$ along the geodesic flow 
 $$G^t: T^* M\To T^* M,$$
 $$C(\infty)=\lim_{t\To +\infty}\inf_{\rho\in T^* M}\frac 1t\int_0^t a(G^s\rho)ds.$$
 Lebeau also showed, on an example, that it is possible to have a spectral gap
 ($D(0)>0$) but no exponential damping ($\rho=0$)~: this surprising phenomenon is typical of non normal operators.
 
 Markus and Matsaev \cite{MM82} proved an analogue of Weyl's law, also found independently later
 on
 by Sj\"ostrand \cite{Sj00}~:
 \begin{equation}\label{e:Weyl}\sharp\left\{ n,0\leq \Re e (\tau_n)\leq \lambda
\right\}=
 \left(\frac{\lambda}{2\pi}\right)^d\left(\int_{p^{-1}]0, 1[}dxd\xi +\cO(\lambda^{-1})  \right)
 \end{equation}
 where $d$ is the dimension of $M$, $p$ is the principal symbol of $-\lap$, namely the function
 $p(x, \xi)=g_x(\xi, \xi)$ defined on the cotangent bundle $T^* M$, and $dxd\xi$ is the Liouville measure
 on $T^* M$ (coming from its canonical symplectic structure).
 
 It is a natural question to ask about the distribution of the imaginary parts $\Im m(\tau_n)$ in the interval
 $[\inf a, \sup a]$. For non normal operators, obtaining fine information on the distribution of the spectrum
 is much harder than for normal operators -- one of the reason being the absence of continuous (or even smooth) functional calculus. Another related difficulty is that there is no general relation between the norm of the resolvent and the distance to the spectrum. Some techniques have being developed to obtain upper bounds on the density of eigenvalues, but lower bounds are notoriously more difficult.
 
 Assuming that the geodesic flow is ergodic with respect to the Liouville measure on an energy layer, Sj\"ostrand proved that most of the $\tau_n$ have asymptotically an imaginary part
 $\Im m(\tau_n)\sim\bar a$, where $\bar a$ denotes the average of $a$ on the energy layer $S^* M=\left\{(x, \xi)\in T^* M,g_x(\xi, \xi)=1\right\} $
 with respect to the Liouville measure~:
 \begin{thm}\cite{Sj00} Assume that the geodesic flow is ergodic with respect to the Liouville measure on $S^*M$. For every $C>0$, for every $\eps >0$, we have
 $$\sharp\left\{ n,\lambda \leq \Re e (\tau_n)\leq \lambda+C, \Im m(\tau_n)\not\in[\bar a-\eps, \bar a+\eps]\right\}=
  o(\lambda^{d-1})$$
  as $\lambda$ goes to infinity.
 \end{thm}
 
 \begin{rem} If $C$ is not too small, one sees from \eqref{e:Weyl} that there exists $\tilde C>0$
 such that
 $$\sharp\left\{ n,\lambda \leq \Re e (\tau_n)\leq \lambda+C\right\}\geq
 \tilde C \lambda^{d-1}$$
 for large $\lambda$.
  Thus, the theorem does say that a majority of the $\tau_n$ have $\Im m(\tau_n)\in
 [\bar a-\eps, \bar a+\eps]$.
 \end{rem}
 
 \begin{rem} More generally, without the ergodicity assumption, \cite{Sj00} proves the following results.
 Introduce the functions on $T^* M$, $$\langle a\rangle_T=\frac{1}{T}\int_{-T/2}^{T/2}a\circ G^s ds,$$
$$\langle a\rangle_\infty=\lim_{T\To+\infty}\langle a\rangle_T,$$ and the real numbers
 $$a_+=\lim_{T\To+\infty}\sup_{S^*M}\langle a\rangle_T,$$
 $$a_-=\lim_{T\To+\infty}\inf_{S^*M}\langle a\rangle_T.$$
 The function $\langle a\rangle_\infty$ is well defined Liouville--almost--everywhere, by the Birkhoff theorem.
 
 Lebeau \cite{Leb93} proves that for any $\eps>0$, there are at most finitely many $n$ with $\tau_n\not\in [a_--\eps, a_++\eps]$, and Sj\"ostrand \cite{Sj00} proves that
 $$\sharp\left\{ n,\lambda \leq \Re e (\tau_n)\leq \lambda+C, \Im m(\tau_n)\not\in[ess\inf \langle a\rangle_\infty -\eps,ess\sup \langle a\rangle_\infty+\eps]\right\}=
  o(\lambda^{d-1}).$$
 \end{rem}
 
  \subsection{Semiclassical formulation\label{p:semiclass}}
  As in \cite{Sj00} we reformulate the problem using semiclassical notations.
In doing so, we also generalize a little the problem. We 
introduce a semiclassical parameter $0<\hbar\ll 1$ and will be interested in the eigenvalues $\tau$
such that $\hbar\tau\Lim_{\hbar\To 0}1$.
If we put $\tau=\frac{\lambda}\hbar$ with
  $\lambda$ close to $1$, then equation \eqref{e:SDWE} can be rewritten as 
  \begin{equation}
 \label{e:SDWEh}\left(-\frac{\hbar^2\lap}2-\frac{\lambda^2}2+i\hbar\lambda a\right)u=0,
 \end{equation}
or \begin{equation}\label{e:spec}(\cP-z)u=0
\end{equation}
if we put $z=\frac{\lambda^2}2$,
and \begin{equation}\label{e:genop}\cP=\cP(z)= P+i\hbar Q(z),\qquad  P=-\frac{\hbar^2\lap}2, \qquad  Q(z)=a\sqrt z.\end{equation}The parameter $z$ is close to $E=\frac12$.

More generally, we will consider a spectral problem of the form
\eqref{e:spec}
where $$\cP=\cP(z)= P+i\hbar Q(z),\qquad  P=-\frac{\hbar^2\lap}2,$$
$z\in \Omega=e^{i]-s_0, s_0[}]E_{\min}, E_{\max}[$, with $0<E_{\min}<\frac12<E_{\max}<+\infty$, $0<s_0<\frac\pi{4}$. We will assume that
 $ Q(z)\in \Psi DO^1 $ is a pseudodifferential operator that depends holomorphically on $z\in \Omega$, and that $ Q$ is formally self-adjoint for $z$ real (the definition of our operator classes is given in Section \ref{s:symbols}).
 
We denote $$\Sigma=\Sigma_\hbar=\left\{z\in\Omega, \,\exists u, (\cP(z)-z)u=0\right\}.$$

The operator $ P$ has principal symbol $p_o(x, \xi)=\frac{g_x(\xi, \xi)}2$, and $ Q(z)$ has principal symbol $q_z(x, \xi)$, taking real values for $z$ real. In these notations, the previous results read as follows~: for any $E_{\min}<E_1\leq E_2<E_{\max}$,
\begin{equation}\label{e:Weyl2}\sharp\left\{ z\in\Sigma, E_1\leq \Re e(z)\leq E_2\right\}=\frac{1}{(2\pi\hbar)^d}\left[\int_{p_o^{-1}[E_1, E_2]}dxd\xi+\cO({\hbar})\right]
\end{equation}
(uniformly for all $c>0$ and for all $E_1, E_2$ such that $|E_2-E_1|\geq 2c\hbar$, $E_1$ and $E_2$
staying away from $E_{\min}, E_{\max}$).
One can show that $\frac{\Im m(z)}{\hbar}$ (for $z\in\Sigma$) has to stay bounded if $E_1, E_2$ stay in a bounded interval.
Fix some $c>0$, and take $E_1=E-c\hbar$ and $E_2=E+c\hbar$.
Let us denote 
$$q^-_E=\lim_{T\To+\infty}\inf_{p_o^{-1}\left\{E\right\}}\langle q_E\rangle_T,$$
$$q^+_E=\lim_{T\To+\infty}\sup_{p_o^{-1}\left\{E\right\}}\langle q_E\rangle_T,$$
then we have \cite{Sj00}
\begin{equation}\label{e:infsup}q^-_E+o(1)\leq\frac{\Im m(z)}{\hbar}\leq q^+_E+o(1)
\end{equation}
for all $z\in\Sigma$ such that $E-c\hbar\leq \Re e(z)\leq E+c\hbar$.
Finally, denote $\bar q_E$ the average of $q_E$ on the energy layer
$p_o^{-1}\left\{E\right\}$. Assuming that the geodesic flow is ergodic on $p_o^{-1}\left\{E\right\}$, then for any $\eps>0$, any $c>0$,
\begin{equation}\label{e:LGN}\sharp\left\{ z\in\Sigma, E-c\hbar\leq \Re e(z)\leq
 E +c\hbar, \frac{\Im m(z)}{\hbar}\not\in [\bar q_E-\eps, \bar q_E+\eps]\right\}=o(\hbar^{1-d}).
\end{equation}

The method of \cite{Sj00} allows to treat the case of a more general $ P$ (and thus a more general
Hamiltonian flow than the geodesic flow), and also to deal with the case when the flow is not ergodic. However, in this paper we will stick to the case of an ergodic geodesic flow; we will add even stronger assumptions in the next paragraph.

 \subsection{Questions, and a few results}
We try to give (partial) answers to the three natural questions~:

{\bf(Q1) (Fractal Weyl law)} Let $I$ be an interval that does not contain $\bar q_E$. Can we describe in a finer way the asymptotic behaviour for $$\sharp\left\{ z, E-c\hbar\leq \Re e(z)\leq E +c\hbar, \frac{\Im m(z)}\hbar\in I\right\} \;?$$
For instance, can we find
$$\lim_{\hbar\To 0} \frac{\log \sharp\left\{ z\in\Sigma, E-c\hbar\leq \Re e(z)\leq E +c\hbar, \frac{\Im m(z)}\hbar\in I\right\} }{\log\hbar} \;?$$

{\bf(Q2) (Inverse problem)}  Suppose we know everything about $ P$
and about the dynamics of the geodesic flow, but that $ Q$ is unknown. To what extent does the knowledge of the eigenfrequencies $\left\{ z\in\Sigma, E-c\hbar\leq \Re e(z)\leq E +c\hbar\right\}$ determine
$q_E$ ?

Replacing $\cP$ by $ B^{-1}\cP B$, where $ B$
is an elliptic pseudodifferential operator with positive principal symbol $b$, amounts to replacing $q$
by $q-\left\{p_o, \log b\right\}$, where $\left\{., .\right\}$ stands for the Poisson bracket on $T^* M$.
Two smooth functions $f$ and $g$ on $T^*M$ are said to be cohomologous (with respect to the geodesic flow) if there exists a third smooth function $h$ such that $f=g+\left\{p_o, h\right\}$. This defines an equivalence relation, we write
$f\sim_{p_o} g$.

It is thus more natural to ask~: 

{\bf(Q2')} Does the knowledge of the eigenfrequencies $\left\{ z\in\Sigma, E-c\hbar\leq \Re e(z)\leq E +c\hbar\right\}$ determine the {\em cohomology class} of
$q_E$ ?

If the length spectrum of $M$ is simple, then one can most probably use a trace formula and recover from the knowledge of $\Sigma$
all the integrals of $q_E$ along closed geodesics. And this is enough to determine the cohomology class of $q_E$ if $M$ has negative sectional curvature~: the Livshitz theorem \cite{Liv71, GK80} says that if two functions have the same integrals along all closed geodesics, then they are cohomologous. Thus,
the answer to (Q2) is probably {\em ``yes''} if $M$ has negative sectional curvature and
the length spectrum is simple (this last assumption is satisfied generically, but unfortunately not for surfaces of constant negative curvature). 

In fact, it also makes sense to ask whether {\em some} knowledge of the imaginary parts alone $\left\{ \Im m(z), z\in\Sigma, E-c\hbar\leq \Re e(z)\leq E +c\hbar\right\}$ allows
to recover {\em some} information about $q_E$. For instance, one can ask the apparently simple question~:\\
{\bf(Q3) (Very weak inverse problem)} Let $C$ denote a constant function.
As follows from \eqref{e:infsup}, we have the implication
$$q_E\sim_{p_o} C \mbox{ on }p_o^{-1}\left\{E\right\}\Longrightarrow \frac{\Im m(z)}\hbar\Lim_{\hbar\To 0, z\in\Sigma, E-c\hbar\leq \Re e(z)+c\hbar} C.$$Does the converse hold ?

\vspace{.8cm}

{\bf  Main assumption on the manifold $M$~:} From now on, we assume that $M$ has constant sectional curvature $-1$. This implies the ergodicity of the geodesic
 flow on any energy layer (with respect to the Liouville measure), and in fact a very chaotic behaviour of trajectories (see Section \ref{s:chaos}). We will indicate how to generalize our results in the case of {\em surfaces} of variable negative curvature; however, it is not clear what to do in higher dimension and variable negative curvature.

\vspace{.8cm}

In the following theorem, $h_{KS}$ stands for the Kolmogorov--Sinai entropy, or metric entropy.
It is an affine functional, taking nonnegative values, defined on $\cM$, the set of $G$--invariant probability measures on $T^*M$~:
see for instance \cite{KH} for the definition of $h_{KS}$. We will also denote $\cM_E\subset\cM$ the set of invariant probability measures carried by $p_o^{-1}\left\{E\right\}$. Since $p_o$ is homogeneous it is enough to consider, for instance,
the case $E=\frac12$, and $p_o^{-1}\left\{E\right\}=S^*M$. For $\mu\in\cM_{\frac12}$, we have $h_{KS}(\mu)\leq d-1$, with equality if and only if $\mu$ is the Liouville measure. We now fix $E=\frac12$ and we denote $q=q_{\frac12}$,
$\bar q=\bar q_{\frac12}$, $q^+=q^+_{\frac12}$, $q^-=q^-_{\frac12}$.

We fix some $c>0$ and denote 
$$\Sigma_{\frac12}=\left\{ z\in\Sigma, \frac12-c\hbar\leq \Re e(z)\leq \frac12 +c\hbar\right\}.$$

\begin{thm} \label{t:upperWeyl} Assume $M$ has constant sectional curvature $-1$. Define
$$H(\alpha)=\sup\left\{ h_{KS}(\mu), \mu\in \cM_{\frac12}, \int q\,d\mu=\alpha\right\}.$$

If $\alpha\geq \bar q$, then for any $c>0$
$$\limsup_{\hbar\To 0} \frac{\log\sharp\left\{ z \in \Sigma_{\frac12}, \frac{\Im m(z)}\hbar\geq \alpha\right\} }{\abs{\log\hbar}} \leq H(\alpha).$$
If $\alpha\leq \bar q$, then for any $c>0$
$$\limsup_{\hbar\To 0} \frac{\log\sharp\left\{ z \in \Sigma_{\frac12},  \frac{\Im m(z)}\hbar\leq \alpha\right\} }{\abs{\log\hbar}} \leq H(\alpha).$$
\end{thm}

\begin{rem} (see \cite{Lal89-II}, \S 4, or \cite{BabLed98}, \S 3 for the argument) The function $H$ is concave and is identically $ -\infty$ outside $[q^-, q^+]$. 
It is continuous and strictly concave in $[q^-, q^+]$, and real analytic in $]q^-, q^+[$
(note that $q^-=q^+$ if and only if $q$ is cohomologous to a constant on $S^*M$).
The function $H$ reaches its maximum $d-1$ at the point $\bar q$, and has finite  (nonnegative) limits at the endpoints $q^-, q^+$.  In particular $H$ is positive in the open interval $]q^-, q^+[$.
\end{rem}

\begin{rem} The key fact in the proof of Theorem \ref{t:upperWeyl}
is the large deviation estimate from \cite{Kif90},
$$\lim_{T\To +\infty}\frac{\log \, L_{\frac12}\left\{\rho\in S^* M, \la q\ra_T(\rho)\in I\right\}}{ T}=\sup\left\{H(\alpha), \alpha\in I\right\}-d-1$$
for any interval $I\subset \IR$ -- where $L_{\frac12}$ is the Liouville measure on $p_o^{-1}\left\{\frac12\right\}=S^*M$. This result gives the volume of the set of trajectories deviating from the Birkhoff ergodic theorem.
See Section \ref{s:chaos}.
\end{rem}

On a surface of variable negative curvature, we can generalize the result to~:
\begin{thm}  \label{t:vari} Assume $M$ is a surface of variable negative curvature. Denote $\varphi$ the infinitesimal
unstable Jacobian (see Section \ref{s:chaos}). Define
$$H(\alpha)=\sup\left\{ \frac{h_{KS}(\mu)}{\int\varphi\,\,d\mu}, \mu\in \cM_{\frac12}, \int q\,d\mu=\alpha\right\}.$$
Then the same conclusion as in Theorem \ref{t:upperWeyl} holds.
 \end{thm}

\begin{rem} For a manifold of variable negative curvature and dimension $d$ we would expect
the same to hold with
$H(\alpha)=(d-1)\sup\left\{ \frac{h_{KS}(\mu)}{\int\varphi\,\,d\mu}, \mu\in \cM_{\frac12}, \int q\,d\mu=\alpha\right\}.$ However, our proof does not work in this case.
 \end{rem}

One may wonder if the $\limsup$ in Theorem \ref{t:upperWeyl} is also a $\liminf$. This question is reminiscent of the fractal Weyl law conjecture asked by Zworski and Sj\"ostrand, but in our situation we can say with certainty that it is not true. Worse, one cannot expect the lower bound to hold for a ``generic'' $q$. In a paper in preparation, Emmanuel Schenck\footnote{Private communication.} is proving that there exists $\delta>0$
such that $\frac{\Im m(z)}\hbar\leq q^+-\delta$ for all $z\in\Sigma_{\frac12}$, provided a certain criterion 
on $q$ is satisfied.  The criterion reads ${\rm Pr}(q-\frac{d-1}2)<q^+$, where the pressure functional ${\rm Pr}$ is defined on the space of continuous functions on $S^*M$ and is the Legendre transform of $-h_{KS}$ (see Section \ref{s:chaos}). The functional ${\rm Pr}$ is lipschitz with respect to the $C^0$ norm, and the condition ${\rm Pr}\left(q-\frac{d-1}2\right)<q^+$ defines a non--empty open 
set in the $C^0$ topology.
For such a $q$ we cannot have
$$\liminf_{\hbar\To 0} \frac{\log\sharp\left\{ z\in \Sigma_{\frac12},  \frac{\Im m(z)}\hbar > q_+-\delta\right\} }{\abs{\log\hbar}} \geq H( q_+-\delta),$$
since $H$ is positive in $]q_-, q_+[$ but the $\liminf $ on the left-hand side is $-\infty$.

In Sections \ref{s:q3} and \ref{s:arithm}, we investigate Question 3 in some special cases.
We work on compact hyperbolic surfaces ($d=2$), and we study operators of the form
$$\lap_\omega f=\lap f-2\la \omega, df\ra+\norm{\omega}_x^2f,$$
called ``twisted laplacians'' -- here $\omega$ is a harmonic real-valued $1$-form on $M$.
Studying the large eigenvalues of $\lap_\omega$ amounts to studying a fixed spectral window for the semiclassical twisted laplacian
$$-\hbar^2\frac{\lap_\omega}2=-\hbar^2\frac{\lap}2+\hbar^2 \la \omega, d.\ra-\hbar^2\frac{\norm{\omega}_x^2}2, \qquad \hbar\To 0.$$
This question falls exactly into the setting of \S \ref{p:semiclass}, with $q(x, \xi)=\la \omega_x,\xi\ra$.
 Since $q(x, -\xi)=-q(x, \xi)$, we have $\bar q=0$. We will denote ${\rm Pr}(\omega)={\rm Pr}(q)$
 the pressure of the function $q$, defined in Section \ref{s:pressureetal}. It will also be interesting to note that $q^+=-q^-$ coincides with the stable norm $\norm{\omega}_s$ defined in Section \ref{p:proof3}, 
that $\norm{\omega}_s+1\geq {\rm Pr}(\omega)\geq \norm{\omega}_s$ and that
${\rm Pr}(t\omega)-|t|\norm{\omega}_s\Lim_{t\To \infty}0$ in the case of surfaces.

\begin{thm}\label{t:q3arithm} Assume $M$ is a compact arithmetic surface coming from a quaternion algebra. Take $\omega\not=0$. Fix $\beta\in(0, 1]$, and $0<\eps<\beta$. Then, for every $\hbar $ small enough, there exists 
$z\in Sp (-\hbar^2\frac{\lap_\omega}2)$ with $|\Re e(z)-\frac12|\leq\hbar^{1-\beta}$, such that
$$ \frac{\Im m(z)}{\hbar} \geq  {\rm Pr}(\omega)-\frac12 -\frac{1+\eps}{2\beta}.$$

Equivalently, given $\beta\in(0, 1]$, and $0<\eps<\beta$, for all $T$ large enough, 
there exists $r_n$ such that $|\Re  e(r_n)-T|\leq T^\beta$ and
$$|\Im m(r_n)|\geq {\rm Pr}(\omega)-\frac12 -\frac{1+\eps}{2\beta},$$
 where $r_n$ is the ``spectral parameter'' defined by $\lambda_n=-(\frac14+r_n^2)$.
\end{thm}
Of course, we deduce immediately the following corollary~:
\begin{coro}
$\sharp\{n, |\Re  e(r_n)|\leq T, |\Im m(r_n)|\geq {\rm Pr}(\omega)-\frac12 -\frac{1+\eps}{2\beta}\}\geq T^{1-\beta},$
asymptotically as $T\To +\infty.$ 
\end{coro}
Since ${\rm Pr}(\omega)\To 1$ as $\omega\To 0$ (but ${\rm Pr}(\omega)\To +\infty$ as $\norm{\omega}_s\To +\infty$)
this result is only interesting if $\norm{\omega}_s$ is large enough (depending on $\beta$).
Note also that if $\norm{\omega}_s$ is large enough we have ${\rm Pr}(\omega)< \norm{\omega}_s+\frac12$,
so that the work of Emmanuel Schenck mentioned above will show that there is a strip $\{\Im m(z)\geq \norm{\omega}_s-\delta\}$ ($\delta>0$) that contains no $r_n$.

The arithmetic case is special, in that the lengths of closed geodesics are well separated~: we can write a trace formula, find a lower bound on the geometric part (despite of its oscillatory nature) and deduce a lower bound on the imaginary parts of eigenvalues. The ideas are borrowed from \cite{Hej}. 

\begin{rem} Another way of defining the twisted laplacian is to write $M=\Gamma\backslash \IH$, where
$\IH$ is the universal cover of $M$ and $\Gamma$ is a discrete group of isometries of $\IH$; to fix an origin $o\in \IH$, and to write
$$\lap_\omega f(x) =e^{\int_o^x\omega}\circ \lap\left( e^{-\int_o^x\omega}f(x)\right);$$
this operator preserves $\Gamma$-periodic functions, hence descends to $M$. The case where $\omega$ takes purely imaginary values is self-adjoint, and analogous to the study of ``Bloch waves''.
The case where $\omega$ takes real values is no longer self-adjoint.
\end{rem}

\begin{rem} Our motivation for working with twisted laplacians on hyperbolic manifolds was
that it is a case where the trace formula is exact. There was, {\em a priori}, hope to prove finer results than in cases where the trace formula is not exact (in the latter case the space of test functions to which the formula can be applied is more limited). {\em A posteriori}, we never use any wild test function that wouldn't be allowed in the general case. So, one can think that the same result holds
for our general operator \eqref{e:genop} -- provided one proves a semiclassical trace formula first.
\end{rem}
If we don't assume arithmeticity, we can only treat the following operator~:
$$-\hbar^2\frac{\lap_{\theta(\hbar)\omega}}2=-\hbar^2\frac{\lap}2+\hbar^2\theta(\hbar) \la \omega, d.\ra-\hbar^2\theta(\hbar)^2\frac{\norm{\omega}_x^2}2, \qquad \hbar\To 0,$$
where $\theta(\hbar)\geq |\log\hbar|$ and $\hbar\theta(\hbar)\To 0$. In other words the non-selfadjoint
perturbation is stronger than in the previous cases.
\begin{thm}\label{t:q3} Assume $M$ is a compact hyperbolic surface. Take $\omega\not=0$. 
Take any function $f(\hbar)\gg \theta(\hbar)^{3/2}\log\log\hbar^{1/2}.$ Then there is a sequence
$\hbar_n\To 0$ such that
$$ \sup\left\{\frac{\Im m(z)}{\hbar_n\theta(\hbar_n)}, z\in Sp\left(-\hbar_n^2\frac{\lap_{\theta(\hbar_n)\omega}}2\right) ,
|\Re e z-\frac12|\leq \hbar_n f(\hbar_n)\right\} \Lim_{n\To +\infty}\norm{\omega}_s,$$
the stable norm of $\omega$.
\end{thm}
See Section \ref{p:proof3} for the definition of the stable norm. Note that with our previous notations, $\norm{\omega}_s=q^+=-q^-$, and ${\rm Pr}(t\omega)-|t|\norm{\omega}_s\Lim_{t\To \infty}0 $ on a hyperbolic surface.

\vspace{.5cm}
{\bf Acknowledgements~:} This work was partially supported by the grant ANR-05-JCJC-0107-01.
I am grateful to UC Berkeley and the Miller Institute for their hospitaly in the spring of 2009.
I thank Dima Jakobson for suggesting that some of the ideas contained in Hejhal's book \cite{Hej} could be used to find lower bounds on the density of eigenvalues in arithmetic situations. In fact, at the same time as this paper was written, Jakobson and Naud managed
to apply these ideas to study the resonances of certain arithmetic convex cocompact surfaces \cite{JN09}.

\section{Note on the definition of the spectrum and its multiplicity\label{s:spec}.}
By the term ``spectrum'', we mean the set $\Sigma$ of $z\in\Omega$ such that $\cP(z)-z$ is not
bijective $H^2(M)\To H^0(M)$. If it is bijective, then the inverse must be continuous, by the closed graph theorem.

If $z$ is restricted to a compact subset of $\Omega$, it is easy to see that $G(z)=I+\lambda^{-1}(\cP(z)-z)$ is invertible for $\lambda>0$ large enough. The inverse $G(z)^{-1}$ is then a compact operator on $H^0(M)$. Besides, one sees that $\cP(z)-z$ is not
bijective $H^2(M)\To H^0(M)$ if and only if $1$ is in the spectrum of $G(z)^{-1}$, if and only if there
exists $u\in H^2(M)$ such that $(\cP(z)-z)u=0$. This shows, in particular, that the ``spectrum'' is discrete
and corresponds to the existence of ``eigenfunctions''.

To define the multiplicity of $z_0\in\Sigma$, we proceed the same way as in \cite{Sj00}. By the density
of finite rank operators in the space of compact operators \cite{RS}, one shows that in a neighbourhood of $z_0$ there exists
a finite rank operator $K(z)$, depending holomorphically on $z$, such that $\cP(z_0)-z_0+K(z_0)$
is invertible. The multiplicity of $z_0$ is then defined as the order of $z_0$ as a zero of the 
holomorphic function
$$z\mapsto \det[(\cP(z)-z+K(z))^{-1}(\cP(z)-z)]=\det[I- (\cP(z)-z+K(z))^{-1}K(z)].$$
It is shown in \cite{Sj00} that this definition does not depend on the choice of $K(z)$, and coincides
with other usual definitions of the multiplicity. Besides, the argument can be extended to the case where $K(z)$ is trace class.

\section{A few facts on the geodesic flow on a negatively curved manifold \label{s:chaos}}
\subsection{Anosov property}
If $M$ has negative sectional curvature, then the geodesic flow on $S^*M$ has the Anosov property \cite{An67}.
This means there are $C, \lambda >0$ such that for each 
$\rho\in S^*M$, the tangent space $T_\rho(S^*M)$ splits into $$
T_\rho(S^*M)=E^u(\rho)\oplus E^s(\rho) \oplus \IR\,X(\rho)\,
$$
where\\
-- the vector field $X$ generates the geodesic flow $G^t$;\\
-- $E^s$ is the stable subspace~: for all $v\in E^s(\rho)$, and for $t\geq 0$, $\norm{DG^t_\rho.v}\leq Ce^{-\lambda t}\norm{v}$;\\
-- $E^u$ is the unstable subspace~:  for all $v\in E^u(\rho)$, and for $t\leq 0$, $\norm{DG^t_\rho.v}\leq Ce^{\lambda t}\norm{v}$.

If $M$ has constant negative curvature $-1$, any $\lambda<1$ will do. We take $\lambda=1-\eps$,
with $\eps$ arbitrarily small.
One also has an upper bound $\norm{DG^t_\rho.v}\leq Ce^{(1+\eps) |t|}\norm{v}$ for any $\eps>0$ and any $t\in \IR$.

\subsection{Pressure, entropy, and large deviation.\label{s:pressureetal}}
The pressure is defined on $C^0(S^* M)$, as the Legendre transform of the entropy~:
$${\rm Pr}(f)=\sup\left\{h_{KS}(\mu)+\int f\,d\mu, \mu\in \cM_{\frac12}\right\}.$$
If $f$ is H\"older, then the supremum is attained for a unique $\mu$, called the equilibrium measure
of $f$. The functional ${\rm Pr}$ is analytic on any Banach space of sufficiently regular functions --
for instance, a space of H\"older functions \cite{BR75, Ru}. Besides, the restriction of ${\rm Pr}$ to any line $\left\{f+tg, t\in\IR\right\}$ is strictly convex, unless $g$ is cohomologous to a constant \cite{Ra73}. If $g$ is sufficiently smooth, we recall that this means that $g=\left\{p_o, h\right\}+c$ for some smooth function $h$ and a constant $c$. If $g$ if H\"older, it is better to use the integral version of the notion. If $\gamma(t)$ is a periodic trajectory of the geodesic flow on $S^*M$ (equivalently, a closed geodesic), 
we denote $l_\gamma$ its period (equivalently, the length of the closed geodesic). We denote $\,d\gamma$ the measure $\int g\,d\gamma=\int_0^{l_\gamma}g(\gamma(t))dt$ on $S^*M$, and $\,\,d\mu_\gamma$ the probability measure $\int g\,\,d\mu_\gamma=l_\gamma^{-1}\int g\,d\gamma$. One says that $g$ is cohomologous to the constant function $c$ if $\int g\,d\gamma=c\,l_\gamma$ for all periodic trajectories of the geodesic flow (the Livschitz theorem says that both notions are equivalent for smooth functions).

Let us now fix a smooth function $q$ on $S^*M$, not cohomologous to a constant. For $\alpha\in\IR$, define
$$H(\alpha)=\sup\left\{h_{KS}(\mu), \mu\in\cM_{\frac12}, \int q\,d\mu=\alpha\right\},$$
$$P(\beta)={\rm Pr}(\beta q)=\sup\left\{h_{KS}(\mu)+\beta\int q\,d\mu, \mu\in \cM_{\frac12}\right\}=\sup_\alpha \alpha\beta+H(\alpha) .$$
The function $H$ is concave, continuous on the interval $[q_-, q_+]$ defined earlier~:
$$q^-=\lim_{T\To+\infty}\inf_{p_o^{-1}\left\{\frac12\right\}}\langle q\rangle_T,$$
$$q^+=\lim_{T\To+\infty}\sup_{p_o^{-1}\left\{\frac12\right\}}\langle q\rangle_T.$$
In the case of a negatively curved manifold, this definition coincides with
$$q^-=\inf\left\{\int q\,d\mu, \mu\in \cM_{\frac12}\right\},$$
$$q^+=\sup\left\{\int q\,d\mu, \mu\in \cM_{\frac12}\right\}.$$
The function $H$ is real analytic and strictly concave in $]q_-, q_+[$. The function $P$ is real analytic, strictly convex on $\IR$.  
Clearly, $P(\beta)\geq \beta q_+$ for $\beta\geq 0$, and it is not very difficult to show that the limit $\lim_{\beta\To +\infty }P(\beta)-\beta q_+$ exists and is nonnegative\footnote{This limit is equal to $H(q^+)$.}. Similarly, 
$P(\beta)\geq \beta q_-$ for $\beta\leq 0$, the limit $\lim_{\beta\To -\infty }P(\beta)-\beta q_-$ exists and is nonnegative. 

The pressure and the entropy appear naturally when studying large deviations for the Birkhoff averages of the function $q$. Denote $J_t(\rho)$ the Jacobian of $DG^t$ going from $E^u(\rho)$ from $E^u(G^t\rho)$. Define $\varphi$ (the infinitesimal unstable Jacobian) by
$$\varphi(\rho)={\frac{dJ_t}{dt}}_{|t=0}(\rho).$$
On a manifold of dimension $d$ and constant negative curvature $-1$, the function $\varphi$ is constant,
equal to $d-1$. In general one can only say that it is a H\"older function. The function
$\varphi$ is not necessarily positive, but it is cohomologous to a positive function, for instance
$\la \varphi\ra_T$ for $T$ large enough. In what follows, we will assume without loss of generality that
$\varphi>0$.

 The two following large deviation results are due to Kifer \cite{Kif90}.
 
 \begin{thm}
\cite{Kif90}, Prop 3.2. 
Let $M$ be a compact manifold of negative sectional curvature. Let $q$ be a smooth function on $S^* M$. For $T>0$, define the function $\la q\ra_T=\frac{1}T\int_{-T/2}^{T/2}q\circ G^s ds$ on $S^*M$. Denote $L_{\frac12}$ the Liouville measure
on $S^*M$.

 Then  $$\lim_{T\To+\infty}\frac{\log \int_{S^* M} e^{T\la q\ra_T(\rho)}dL_{\frac12}(\rho)}{ T} =
 {\rm Pr}(q-\varphi).$$

\end{thm}

As a consequence, one also has~: 
\begin{thm}\label{p:LDP} \cite{Kif90}, Thm 3.4 (i)
Let $M$ be a compact manifold of negative sectional curvature. Let $q$ be a smooth function on $S^* M$. For $T>0$, define the function $\la q\ra_T=\frac{1}T\int_{-T/2}^{T/2}q\circ G^s ds$ on $S^*M$. Denote $L_{\frac12}$ the Liouville measure
on $S^*M$.

 Then, for any closed interval $I\subset \IR$, we have
 $$\limsup_{T\To+\infty}\frac{\log \, L_{\frac12}\left\{\rho\in S^* M, \la q\ra_T(\rho)\in I\right\}}{ T}\leq \sup\left\{h_{KS}(\mu)-\int \varphi \,d\mu, \mu\in\cM_{\frac12}, \int q\,d\mu \in I \right\}.$$
 
 For any open interval $I\subset \IR$, we have
 $$\liminf_{T\To+\infty}\frac{\log \, L_{\frac12}\left\{\rho\in S^* M, \la q\ra_T(\rho)\in I\right\}}{ T}\geq \sup\left\{h_{KS}(\mu)-\int \varphi \,d\mu, \mu\in\cM_{\frac12}, \int q\,d\mu \in I\right\}.$$

(ii) Let $M$ be a compact manifold of dimension $d$, with constant sectional curvature $-1$. Then (i) can be rephrased as follows. Let $q$ be a smooth function on $S^* M$. For $T>0$, define the function $\la q\ra_T=\frac{1}T\int_{-T/2}^{T/2}q\circ G^s ds$ on $S^*M$. Denote $L_{\frac12}$ the Liouville measure
on $S^*M$.
Then, for any closed interval $I\subset \IR$, we have
 $$\limsup_{T\To+\infty}\frac{\log \, L_{\frac12}\left\{\rho\in S^* M, \la q\ra_T(\rho)\in I\right\}}{ T}\leq \sup\left\{H(\alpha), \alpha\in I\right\}-(d-1)$$
where $H$ is the function $H(\alpha)=\sup\left\{h_{KS}(\mu), \mu\in\cM_{\frac12}, \int q\,d\mu=\alpha\right\}$.
For any open interval $I\subset \IR$, we have
 $$\liminf_{T\To+\infty}\frac{\log \, L_{\frac12}\left\{\rho\in S^* M, \la q\ra_T(\rho)\in I\right\}}{ T}\geq \sup\left\{H(\alpha), \alpha\in I\right\}-(d-1).$$

\end{thm}

The pressure and entropy functions also appear when counting closed geodesics $\gamma$ with a given $q$-average~:
$${\rm Pr}(q)=\lim_{t\To+\infty }\frac1t\log\sum_{\gamma,\, l_\gamma\leq t}e^{\int q\,d\gamma},$$
and as consequence
$$\limsup_{t\To +\infty}\frac1t\log\sharp\left\{ \gamma, \,l_\gamma\leq t, \int q\,\,d\mu_\gamma\in I\right\}\leq
\sup\left\{H(\alpha), \alpha\in I\right\},$$
(for a closed interval $I$),
$$\liminf_{t\To +\infty}\frac1t\log\sharp\left\{ \gamma, \,l_\gamma\leq t, \int q\,\,d\mu_\gamma\in I\right\}\geq
\sup\left\{H(\alpha), \alpha\in I\right\},$$
(for an open interval $I$). See \cite{Kif94}.

In negative variable curvature, we will also need the following variant of Theorem \ref{p:LDP}~:
\begin{thm} \label{t:LDP2}
Let $M$ be a compact manifold of negative sectional curvature. Let $q$ be a smooth function on $S^* M$. Let $\phi$ be a  smooth positive function. Denote $L_{\frac12}$ the Liouville measure
on $S^*M$. For $\rho\in S^* M$ and $t\in\IR$, define $\cT_\rho(t)$ by $\int_0^{\cT_\rho(t)}\phi(G^s\rho)ds=(d-1)t$.  For $t>0$, define the function $$\la q\ra_{\cT(-t/2), \cT(t/2)}=\frac{1}{\cT_\rho(t/2)-\cT_\rho(-t/2)}\int_{\cT_\rho(-t/2)}^{\cT_\rho(t/2)}q\circ G^s (\rho)ds$$ on $S^*M$.

 Then, for any closed interval $I\subset \IR$, we have
 \begin{multline*}\limsup_{t\To+\infty}\frac{\log \, L_{\frac12}\left\{\rho\in S^* M, \la q\ra_{\cT(-t/2), \cT(t/2)}(\rho)\in I\right\}}{ t}\\ \leq (d-1) \sup\left\{\frac{h_{KS}(\mu)}{\int \phi \,d\mu}- \frac{\int \varphi \,d\mu}{\int \phi \,d\mu}, \mu\in\cM_{\frac12}, \int q\,d\mu \in I \right\}.\end{multline*}

 For any open interval $I\subset \IR$, we have
 \begin{multline*}\liminf_{t\To+\infty}\frac{\log \, L_{\frac12}\left\{\rho\in S^* M, \la q\ra_{\cT(-t/2), \cT(t/2)}(\rho)\in I\right\}}{ t}\\ \geq (d-1) \sup\left\{\frac{h_{KS}(\mu)}{\int \phi \,d\mu}- \frac{\int \varphi \,d\mu}{\int \phi \,d\mu}, \mu\in\cM_{\frac12}, \int q\,d\mu \in I \right\}.\end{multline*}
 
\end{thm}

\begin{proof} Define a flow $\bar G$ on $S^*M$ that has the same trajectories as $G$ but different speed~:
$\bar G^t(\rho)=G^{\cT_\rho(t)}(\rho)$.
For the new flow, the infinitesimal unstable Jacobian is equal to $(d-1)\frac{\int \varphi \,d\mu}{\int \phi \,d\mu}$. If $\mu$ is an invariant probability measure of $G$, then
$d\bar\mu=\frac{\phi\,\,d\mu}{\int \phi\,\,d\mu}$ is an invariant probability measure of $\bar G$.
Besides, their entropies are related by the Abramov formula~:
$$h_{KS}(\bar\mu)=(d-1)\frac{h_{KS}(\mu)}{\int \phi\,\,d\mu},$$
where the entropies of $\bar\mu$ and $\mu$ are computed with respect to $\bar G$ and $G$ respectively.

The theorem is then again an application of Theorem 3.4 in \cite{Kif90} for the Anosov flow $\bar G$.
\end{proof}

\section{Averaging\label{s:ave}}
We are now ready to start the proof of Theorem \ref{t:upperWeyl}. We will work in dimension $d$ and constant negative $-1$. The changes to make in order to get Theorem \ref{t:vari} are indicated in Remarks \ref{r:change} and \ref{r:change2}.

The following proposition is proved in \cite{Sj00}, \S 2~:
\begin{prop} Let $T>0$, there exists an invertible selfadjoint pseudodifferential operator $A_T\in \Psi DO^0$
such that
$$A_T^{-1}( P+i\hbar  Q(z))A_T= P+i\hbar\Op_\hbar(q^T(z))+\hbar^2 R_T(z)$$
for $z\in \Omega$;
with $R_T\in\Psi DO^0$ depending holomorphically on $z\in \Omega$, and with $q^T(z)\in S^{1 }$ depending holomorphically on $z\in \Omega$, equal to $\la q\ra_T-q+q_z$ in a neighbourhood
of $p_o^{-1}\left(\frac12\right)$.
\end{prop}
The definition of our symbol classes $S^m $ and operator classes $\Psi DO^m$ is given in Section \ref{s:symbols}.

We recall that the operator $A_T$ constructed by Sj\"ostrand is $A=\Op_\hbar(e^{g_T})$,
where
$$g_T= \frac12\int_0^{T/2}\left(\frac{2s}T-1\right)q\circ G^{s}ds +\frac12\int_{-T/2}^0\left(\frac{2s}T+1\right) q\circ G^{s}
ds$$
on $p_o^{-1}\left(\frac12\right)$. The function $g_T$ solves $\{p_o, g_T\}=q-\la q\ra_T.$
Exactly the same proof yields~:
\begin{prop}\label{p:ave} Assume $M$ has constant curvature $-1$. Let $\eps>0$ and $T=(1-4\eps)|\log\hbar|$. Define $\delta=\frac{1-\eps}2$.
There exists an invertible selfadjoint pseudodifferential operator $A_T \in \Psi DO^0_\delta$
such that
 $$A_T^{-1}( P+i\hbar  Q(z))A_T= P+i\hbar\Op_\hbar(q^T(z))+\hbar^2 R_T(z)$$for $z\in \Omega$;
with $R_T\in \hbar^{-2\delta}\Psi DO_\delta^0$ depending holomorphically on $z$, and with $q^T(z)\in S_\delta^{1}$ depending holomorphically on $z$, equal to $\la q\ra_T-q+q_z$ in a neighbourhood
of $p_o^{-1}\left(\frac12\right)$.\end{prop}
 

In what follows, we will restrict our attention to a region where $|z-\frac12|=\cO(\hbar)$. As a consequence, we can write
\begin{equation} P+i\hbar\Op_\hbar(q^T(z))+\hbar^2 R_T(z)= P+i\hbar\Op_\hbar(q^T)+\hbar \tR_T(z)
\label{e:theope}\end{equation}
with $q^T=q^T\left(\frac12\right)=\la q\ra_T$ in a neighbourhood
of $p_o^{-1}\left(\frac12\right)$, and $\tR_T(z)$ is a pseudodifferential operator depending holomorphically on $z\in \Omega$, tending to zero when $\hbar\To 0$ and $|z-\frac12|=\cO(\hbar)$. More precisely,
$$ \tR_T(z)=\left(z-\frac12\right) Q'(z)+ \hbar R_T(z),$$
$R_T\in \hbar^{-2\delta}\Psi DO_\delta^0$ depending holomorphically on $z$, and 
$Q'(z)\in \Psi DO_\delta^1$ depending holomorphically on $z$.

\begin{rem}\label{r:change} To treat the case of variable curvature, we should modify Proposition \ref{p:ave} as follows. Fix $\phi$ a smooth function such that $  \phi\geq \varphi $. Define $\cT_{\rho}( \frac{T}2)$, $\cT_{\rho}(-\frac{T}2)$ as in Theorem \ref{t:LDP2}, in a neighbourhood of $p_o^{-1}\left(\frac12\right)$.
We have to choose $\phi$ smooth because we want $\cT_\rho$ to depend smoothly on $\rho$.
In dimension $d=2$, we have $\cT_{\rho}( \frac{T}2)\in S^0_\delta$ with $\delta=\frac{1-\eps}2$ (which may not be true for $d>2$ since the unstable Jacobian no longer controls the derivatives of the geodesic flow).
In Proposition \ref{p:ave}, we now define $A_T =\Op_\hbar(e^{g_T})$ where
$$g_T(\rho)= \frac12\int_0^{\cT_\rho(T/2)}\left(\frac{s}{\cT_\rho(T/2)}-1\right)q\circ G^{s}ds +\frac12\int_{-\cT_\rho(T/2)}^0\left(\frac{s}{\cT_\rho(T/2)}+1\right) q\circ G^{s}
ds$$
on $p_o^{-1}\left(\frac12\right)$. In the last sentence of Proposition \ref{p:ave}, we replace $\la q\ra_T$ by $\la q\ra_{\cT(-\frac{ T}2), \cT( \frac{T}2)}+r_T$ where $r_T=q-\{p_o, g_T\}-\la q\ra_{\cT(-\frac{ T}2), \cT( \frac{T}2)}$ satisfies $r_T\in |\log\hbar|^{-1}S^0_\delta$ and $\{p_o, r_T\}\in |\log\hbar |^{-1}S^0_\delta$.  For $d=2$, all the operators $A_T$, $R_T$ {\em etc} stay in the same class as stated in Proposition \ref{p:ave}.
\end{rem}

\section{Perturbations with controlled trace norm.\label{s:pert}}
In the following sections, we let $z$ vary in a disc of radius $\cO(\hbar)$ around  $\frac12$. We will write $2z=1+\zeta$, $\zeta=\cO(\hbar)$. We consider the operator \eqref{e:theope}, that we write
\begin{equation}\label{e:theope2}
\cP_T=\cP_T(z)= P+i\hbar  Q_T+\hbar \tR_T(z),\qquad  Q_T= \Op_\hbar(q^T).
\end{equation}
 Note that we have
 $$\left\{p_o, q^T\right\}=\cO\left(\frac1T\right)$$
 in a neighbourhood of $p_o^{-1}\left(\frac12\right)$. By Proposition \ref{p:CV},
 this implies
 \begin{equation}\label{e:comm}\norm{[P, Q_T]u}\leq
C \left(\frac{\hbar}T +\cO(\hbar^{2-2\delta})\right)\norm{u}
+\cO(\hbar)\norm{(P-\frac12)u}.\end{equation}
 
We now want to make a small perturbation $\tilde\cP$ of $\cP$ with a good control over the resolvent $(\tilde\cP(z)-z)^{-1}$, and over the trace class norm $\norm{.}_1$ of $\tilde\cP-\cP$.

We construct a pseudodifferential operator $\tQ_T=\Op_\hbar(\tilde q^T)\in \Psi DO^1_\delta$ such that
 $\tilde q^T\leq q^T$ on $p_o^{-1}\left(\frac12\right)$ and $\left\{p_o, \tilde q^T\right\}=\cO\left(\frac1T\right)$. In addition, we fix some $\eps>0$ and introduce an arbitrarily small $\theta>0$, and we want $q^T=\tilde q^T$ on $p_o^{-1}(]\frac12-\eps, \frac12+\eps[)\cap\left\{q^T\leq \alpha-3\theta\right\}$, and $\tilde q^T\leq\alpha-2\theta$ everywhere on $p_o^{-1}(]\frac12-\eps, \frac12+\eps[)$. For instance we can take $\tilde q^T=a(q^T)$ on $p_o^{-1}(]\frac12-\eps, \frac12+\eps[)$ where $a$ is real and smooth, $a(E)\leq E$, $|a'|\leq 1$; $a(E)=E$ if $E\leq \alpha-3\theta$, and $a\leq \alpha-2\theta$ everywhere.
 
 \begin{rem}At this stage it is convenient to choose a positive quantization scheme $\Op_\hbar$, in order to have $\Op_\hbar(q^T)\geq \Op_\hbar(\tilde q^T)$.
 \end{rem}

Let $0\leq f\in \cS(\IR)$, with $\hat f\in C_0^\infty$, where $\hat f(t)=\int e^{itE}f(E)dE$ is the Fourier transform. Put
$$\tilde\cP=P+i\hbar\hat Q_T+\hbar\tR_T(z),$$
with
$$\hat Q_T=Q_T+f\left(\frac{2P-1}\hbar\right)(\tQ_T-Q_T)f\left(\frac{2P-1}\hbar\right).$$

The following proposition is proved in \cite{Sj00} for fixed $T$ (and $\delta=0$, that is, with standard symbol classes). One can follow the proof of \cite{Sj00} line by line and check that it is still valid for $T=(1-4\eps)|\log\hbar|$, $\eps>0$ very small~:
\begin{prop}\label{p:plagiat} Let $P=-\hbar^2\frac{\lap}2$. Let $Q=Q(z)\in \Psi DO^1$ have principal symbol $q(z)$ depending holomorphically on $z\in\Omega$, and be formally self-adjoint when $z$ is real. Let
$$\cP_T=P+i\hbar Q_T+\hbar \tR_T(z),\qquad Q_T=Q_T\left(\frac12\right), \qquad z=\frac{1+\zeta}2, \qquad
\zeta=\cO(\hbar),$$
be the operator defined in \eqref{e:theope2}, with $Q_T=\Op_\hbar(q^T)\in \Psi DO_\delta^1$, and $\tR_T(z)\in \hbar^{1-2\delta}\Psi DO^0_\delta
+(z-\frac12)\Psi DO_\delta^1$. Let $\tQ_T=\Op_\hbar(\tilde q^T)\in \Psi DO^1_\delta$, with $\tilde q^T=a(q^T)$ on $p_o^{-1}(]\frac12-\eps, \frac12+\eps[)$,
where $a$ is real and smooth, $a(E)\leq E$, $|a'|\leq 1$; $a(E)=E$ if $E\leq \alpha-3\theta$, and $a\leq \alpha-2\theta$.

 Put
$$\tilde\cP_T=P+i\hbar\hat Q_T+\hbar \tR_T(z),$$
with
$$\hat Q_T=Q_T+f\left(\frac{2P-1}\hbar\right)(\tQ_T-Q_T)f\left(\frac{2P-1}\hbar\right).$$

Then
$$\norm{\tilde\cP_T-\cP_T}\leq \hbar\left( \norm{f}_\infty^2\sup_{p_o^{-1}\left(\frac12\right)}(q^T-\tilde q^T)+\cO(\hbar^{1-2\delta})\right)$$
and
\begin{multline*}\norm{\tilde\cP_T-\cP_T}_1
\leq C_d\hbar^{2-d}\Big[\hat{f^2}(0)\int_{p_o^{-1}\left(\frac12\right)}(q^T-\tilde q^T)L_{\frac12}(d\rho)\\
+\sum_{k=1}^{N-1}\hbar^k \abs{D^{2k}_t\hat{f^2}(0)}\int_{p_o^{-1}\left(\frac12\right)}\abs{D^{2k}_\rho(q^T-\tilde q^T)}L_{\frac12}(d\rho)+\cO(\hbar^{N(1-2\delta)})\Big],
\end{multline*}
where $D^{2k}_t$ and $D^{2k}_\rho$ are differential operators of degree $\leq 2k$, respectively on
$\IR$ and $T^*M$.

If we restrict $z$ by assuming that for some continuous function $\alpha(E)>0$, defined on some bounded interval $J$ containing $0$, we have
$$\frac{\Im m(\zeta)}{2\hbar}-q^T+f\left(\frac{\Re e(\zeta)}\hbar\right)^2(q^T-\tilde q^T)\geq \alpha\left(\frac{\Re e(\zeta)}\hbar\right), $$ on $p_o^{-1}\left(\frac12\right), \frac{\Re e(\zeta)}\hbar\in J$,\\
then for  $\hbar$ small enough, $(z-\tilde \cP_T)^{-1}$ exists, and we have
$$\norm{\left(\frac1\hbar(z-\tilde \cP_T)\right)^{-1}}\leq
\frac{2+12\sup_{p_o^{-1}\left(\frac12\right)}(q^T-\tilde q^T)\norm{f'}_\infty\norm{f}_\infty}{\alpha\left(\frac{\Re e(\zeta)}\hbar\right)}.$$

\end{prop}
The proof is identical to the proof in \cite{Sj00}. In Appendix \ref{a:plagiat} we will give some details, for the reader's convenience.

\begin{coro} Define 
$$\tilde\Omega_\hbar=\left\{\frac12-2c\hbar\leq \Re e(z)\leq\frac12 +2c\hbar\right\}\cap\left\{
(\alpha-\theta)\hbar\leq \Im m(z)\leq 4\norm{q}_\infty\hbar\right\}\subset\IC.$$ For $z\in\tilde\Omega_\hbar$, the operator $z-\tilde \cP_T$ is invertible, and
$$\norm{(\tilde \cP_T-z)^{-1}}\leq \frac{C_{f, q}}{\theta\hbar}.$$
\end{coro}

\begin{coro}\label{c:maincoro}For $\hbar$ small enough, we have
\begin{eqnarray*}\norm{\tilde\cP_T-\cP_T}_1&\leq &C_{d, f, q}  \hbar^{2-d} L_{\frac12}\left(\left\{\tilde q^T\not=q^T\right\}\cap \left\{p_o^{-1}\left(\frac12\right)\right\}\right)\\
&\leq & C_{d, f, q}  \hbar^{2-d} L_{\frac12}\left(\left\{ q^T\geq \alpha -3\theta\right\}\cap \left\{p_o^{-1}\left(\frac12\right)\right\}\right)\\
&\leq &  C_{d, f, q}  \hbar^{2-d} e^{T[H(\alpha-3\theta)-(d-1)+\eps]}\\
&\leq  & C_{d, f, q} \hbar^{2-d} \hbar^{ [(d-1)-H(\alpha-3\theta)-\eps](1-4\eps)}
\end{eqnarray*}
for $\hbar$ small enough.
\end{coro}

\section{Jensen's inequality\label{s:Jensen}}
 We already defined $$\tilde\Omega_\hbar=\left\{\frac12-2c\hbar\leq \Re e(z)\leq\frac12 +2c\hbar\right\}\cap\left\{
(\alpha-\theta)\hbar\leq \Im m(z)\leq 4\norm{q}_\infty\hbar\right\}\subset\IC.$$
To finish the proof of Theorem \ref{t:upperWeyl}, we also introduce the set $$\Omega_\hbar=
\left\{\frac12-c\hbar\leq \Re e(z)\leq\frac12 +c\hbar\right\}\cap \left\{
\alpha\hbar\leq \Im m(z)\leq 3\norm{q}_\infty\hbar\right\}\subset\tilde\Omega_\hbar.$$
For $z\in\tilde\Omega_\hbar$, we can write
$$ \cP_T-z=(\tilde \cP_T-z)(1+K(z))$$
where $K(z)$ is the trace class operator $(\tilde \cP_T-z)^{-1}(\cP_T-\tilde \cP_T)$.

We can bound the number of eigenvalues of $\cP_T$ in $\Omega_\hbar$ by the number
of zeros of the holomorphic function $g(z)=\det(1+K(z))$ in $\Omega_\hbar$. Let us call $N(g, \Omega_\hbar)$ this number of zeros.
Introduce $z_0=\frac12+2i\hbar \norm{q}_\infty$.
By the Jensen inequality \cite{Rud}, 
\begin{equation}N(g, \Omega_\hbar)\leq C \left(\log\norm{g}_{\infty, \tilde\Omega_\hbar}-\log|g(z_0)|\right),\label{e:jensen}\end{equation}
where the constant $C$ does not depend on $\hbar$ (because the rectangles $\tilde\Omega_\hbar$ and $\Omega_\hbar$ can be transported, by translations and homotheties, to the fixed rectangles
$\tilde\Omega_1$ and $\Omega_1$).

On the one hand, for all $z\in \tilde\Omega_\hbar$,
\begin{eqnarray*}
|\det(1+K(z))|&\leq& \exp\norm{K(z)}_1
\\ &\leq& \exp\left(  \norm{\tilde \cP_T-z)^{-1}}\norm{\cP_T-\tilde \cP_T}_1     \right)\\
&\leq&  \exp\left( \frac{C_{d,f, q}}{\theta\hbar }  \hbar^{2-d} \hbar^{ [(d-1)-H(\alpha-3\theta)-\eps](1-4\eps)}
\right)\\
&\leq&   \exp\left(  C_{f, q, \theta, d} \hbar^{1-d}\hbar^{ [(d-1)-H(\alpha-3\theta)-\eps](1-4\eps)}
\right).
\end{eqnarray*}

On the other hand, we also know that $ \norm{(1+K(z_0))^{-1}}\leq C\hbar^{-1}$~: since $z_0$ has `large'
imaginary part,  $\cP_T-z_0$ is invertible, and it is easy to get a bound $\norm{(\cP_T-z_0)^{-1}}=\cO(\hbar^{-1})$. We use the same calculation as in \cite{Sj00} and get
\begin{eqnarray*}
|\det(1+K(z_0))^{-1}|&=& |\det (1-K(z_0)(1+K(z_0))^{-1})|
\\ &\leq& \exp\norm{K(z_0)(1+K(z_0))^{-1}}_1
\\ &\leq&\exp \norm{K(z_0)}_1 \norm{(1+K(z_0))^{-1}}
\\&\leq&   \exp\left( \tilde C_{f, q, \theta, d} \hbar^{1-d}\hbar^{ [(d-1)-H(\alpha-3\theta)-\eps](1-4\eps)}\right)
\end{eqnarray*}
so that
$$|\det(1+K(z_0)|\geq  \exp\left( -\tilde C_{f, q, \theta, d} \hbar^{1-d}\hbar^{ [(d-1)-H(\alpha-3\theta)-\eps](1-4\eps)}\right).$$
This, combined to \eqref{e:jensen}, yields
$$\abs{N(g, \Omega_\hbar)}\leq C \hbar^{1-d}\hbar^{ [(d-1)-H(\alpha-2\theta)-\eps](1-4\eps)}$$
Since $\theta$ and $\eps>0$ are arbitrary, we have proved Theorem \ref{t:upperWeyl}.

\begin{rem}\label{r:change2} Starting from Remark \ref{r:change}, the proof of Theorem \ref{t:vari} goes exactly along the same lines. We find
$$\limsup_{\hbar\To 0} \frac{\log\sharp\left\{ z \in \Sigma_{\frac12}, \frac{\Im m(z)}\hbar\geq \alpha\right\} }{\abs{\log\hbar}} \leq \tilde H(\alpha),$$
where $\tilde H(\alpha)=(d-1) \sup\left\{\frac{h_{KS}(\mu)}{\int \phi \,d\mu}- \frac{\int \varphi \,d\mu}{\int \phi \,d\mu} +1, \mu\in\cM_{\frac12}, \int q\,d\mu =\alpha\right\}$ and $\phi$ is as in Remark \ref{r:change}. Letting $\phi$ converge to $\varphi$ uniformly, we obtain Theorem \ref{t:vari}.
\end{rem}

\section{About Question 3\label{s:q3}}
In this section, we consider a particular case of the problem \eqref{e:spec} in which the trace formula is exact.
We then try to investigate Question 3 on this example.

Let $M$ be a compact hyperbolic surface~: $M$ can be written as $M=\Gamma\backslash \IH$,
where $\IH$ is the hyperbolic disc and $\Gamma$ is a discrete subgroup of the group of hyperbolic isometries.
Let $[\omega]\in H^1(M, \IC)$ be represented by the harmonic complex valued $1$--form $\omega$. Introduce the twisted laplacian
$$\lap_\omega f=\lap f-2\la \omega, df\ra+\norm{\omega}_x^2f.$$
Studying the large eigenvalues of $\lap_\omega$ amounts to studying a fixed spectral window for the semiclassical twisted laplacian
$$-\hbar^2\frac{\lap_\omega}2=-\hbar^2\frac{\lap}2+\hbar^2 \la \omega, d.\ra-\hbar^2\frac{\norm{\omega}_x^2}2, \qquad \hbar\To 0.$$
The ``usual'' selfadjoint case is when $\omega$ has coefficient in $i\IR$. We shall instead be interested
in the case when $\omega$ has coefficients in $\IR$.
The operator falls exactly into the case studied in \S \ref{p:semiclass}, with $q(x, \xi)=\la \omega_x, \xi\ra$.
The geodesic flow is ergodic, and Sj\"ostrand's result tells us that ``most'' eigenvalues of $-\hbar^2\frac{\lap_\omega}2$ such that $\Re e(z)\in [\frac12-C\hbar, \frac12+C\hbar]$ have imaginary part
$\Im m(z)=o(\hbar)$. Equivalently, ``most'' eigenvalues of $-\lap_\omega$ such that $\Re e(z)\in
[\lambda-C\sqrt{\lambda}, \lambda+C\sqrt{\lambda}]$ have imaginary part $\Im m(z)=o(\sqrt\lambda)$.

We rephrase Question 3 as\\
{\bf (Q3')} If $\omega\not=0$, is it possible to have $\frac{\Im m(z)}\hbar\To 0$ as $\hbar\To 0$ and $\Re e(z)\in [\frac12-C\hbar, \frac12+C\hbar]$, $z\in Sp(-\hbar^2\frac{\lap_\omega}2)$ ?
\vspace{.4cm}
\\
{\bf Conjecture~:} I conjecture the opposite~:
if $\omega\not=0$,  then there is a sequence $\hbar_n\To 0$,  $z_n\in Sp(-\hbar_n^2\frac{\lap_\omega}2)$ with $\Re e(z_n)\in [\frac12-C\hbar_n, \frac12+C\hbar_n]$ and $\frac{\Im m(z_n)}{\hbar_n}\not\To 0$. 

\vspace{.4cm}
As is usual in hyperbolic spectral theory, we introduce the spectral parameter $r$~: if $\lambda_j$ is an eigenvalue of $-\lap_\omega$, we denote $\lambda_j=\frac14+r_j^2$. Yet another way to phrase Question 3 is to ask whether it is possible to have $\Im m(r_j)\To 0$ as $\Re e(r_j)\To \infty$.
Sj\"ostrand's results say that $\Im m(r_j)$ is bounded and that $\Im m(r_j)\To 0$ for a subsequence of density one. But I naturally believe that it is impossible to have $\Im m(r_j)\To 0$ for the whole sequence, 
unless $\omega=0$.

Recall the Selberg trace formula \cite{Sel56}, valid for $\omega\in H^1(M, i\IR)$~:
\begin{equation}
\sum_{\lambda_j} \hat f(r_j)=\frac{Area(M)}{4\pi}\int_{-\infty}^{+\infty}\hat f(r)r\tanh(\pi r) dr+\sum_{\gamma}\frac{e^{\int_\gamma \omega}l_{\gamma_o}}{\sinh\frac{l_\gamma}2}f(l_\gamma),
\end{equation}
for any function $f$ on $\IR$, even, smooth enough, and decaying faster than any exponential.
On the right hand side, the sum runs over the set of closed geodesics (equivalently, the set of conjugacy classes in $\Gamma$). If $\gamma$ is a periodic geodesic, we denote $l_\gamma>0$ its
length, or period; and $l_{\gamma_o}$ is its shortest period.

\begin{prop} The trace formula holds, under the same assumptions on $f$,
if $\omega\in H^1(M, \IR)$.
\end{prop}
Recall that the Fourier transform is defined by $\hat f(r)=\int e^{iru}f(u)du.$

\begin{proof}
 
Take $\omega\in H^1(M, \IR)$.
We consider the operator $\lap_{z\omega}$ for $z\in \IC$. The argument of Section \ref{s:spec}
shows that this operator has discrete spectrum (eigenvalues).
 
The right hand side of the trace formula reads
\begin{equation}\label{e:tracez}\frac{Area(M)}{4\pi}\int_{-\infty}^{+\infty}\hat f(r)r\tanh(\pi r) dr+\sum_{\gamma}\frac{e^{z\int_\gamma \omega}l_{\gamma_o}}{\sinh\frac{l_\gamma}2}f(l_\gamma)\end{equation}
and clearly is an entire function of $z$.

The left hand side is 
$\sum_j \hat f(r_j(z))$. To check that it is an entire function of $z$, we first note that $\hat f(r)$, being
an even entire function, can be written as $g(\frac14+r^2)$ where $g$ is entire. Thus
$\hat f(r_j(z))=g(\lambda_j(z))$, where the $\lambda_j(z)$ are the eigenvalues of $-\lap_{z\omega}$.

If $z$ is restricted to a bounded subset $\Omega$ of $\IC$, 
we note that the $\Im m (r_j(z))$ are uniformly bounded. To that end, we write $-\lambda_j=\frac14+r_j^2$ with $r_n=x+iy$. The eigenvalue equation
$$-\lap_{z\omega}f=\lambda_j f$$
with $f$ normalized in $L^2$
implies both equations
$$\frac14+x^2-y^2=\int |\nabla f|^2+2\beta i\int\la\omega, df\ra\bar f+(\beta^2-\alpha^2)\int 
\norm{\omega}_x^2\abs{f}^2$$
and
$$2xy=-2\alpha i\int\la\omega, df\ra\bar f-2\alpha\beta\int 
\norm{\omega}_x^2\abs{f}^2$$
if we decompose $z=\alpha+i\beta\in \IR+i\IR$ and $r_j=x+iy\in \IR+i\IR$. If $\alpha$ and $\beta$ stay bounded, it also follows that $y$ must stay bounded.

Besides
$$\lambda^{-2}\sharp\left\{n , 0\leq \Re e(r_n(z))<\lambda\right\} $$
is bounded uniformly for $\lambda>1$ and $z$ staying in a compact set of $\IC$ (the arguments of \cite{Sj00}, \S 4, or of our Sections \ref{s:ave}, \ref{s:pert}, \ref{s:Jensen} are locally uniform in $z$). Since $\hat f$ is rapidly decreasing in each horizontal strip, it follows that
the sum 
$\sum_j \hat f(r_j(z))$ is the uniform limit of the partial sums
$\sum_{|\Re e(r_j(z))|<\lambda}\hat f(r_j(z))$. But for a given $\lambda$, this is a holomorphic function of $z$ (in the open set $\{z,\Re e(r_j(z))\not=\lambda \mbox{ for all }j\}$), since $\lap_{z\omega}$ depends holomorphically on $z$ and $\sum_{|\Re e(r_j(z))|<\lambda}\hat f(r_j(z))$ can be defined by holomorphic functional calculus.

This shows that $\sum_j \hat f(r_j(z))$ is also an entire function. Both sides of \eqref{e:tracez}
coincide for $z\in i\IR$ (by the usual trace formula), thus they must coincide everywhere.
\end{proof}

Introduce some parameters $\sigma, R, T>0$, and take
$$f(u)=\frac{1}{\sqrt{2\pi}\sigma}\left[e^{-\frac{(u-T)^2}{2\sigma^2}}e^{iuR}+e^{-\frac{(u+T)^2}{2\sigma^2}}e^{-iuR}\right]$$
so that
$$\hat f(r)=e^{-\frac{\sigma^2}2(r-R)^2}e^{-iTr}+e^{-\frac{\sigma^2}2(r+R)^2}e^{iTr}.$$
This yields~:
\begin{multline}\label{e:tracegauss}
\sum_j e^{-\frac{\sigma^2}2(r_j-R)^2}e^{-iTr_j}+e^{-\frac{\sigma^2}2(r_j+R)^2}e^{iTr_j}\\
=\frac{Area(M)}{4\pi}\int_{-\infty}^{+\infty}r\tanh\pi r \left[e^{-\frac{\sigma^2}2(r-R)^2}e^{-iTr}+e^{-\frac{\sigma^2}2(r+R)^2}e^{iTr}\right]dr\\ +
\sum_\gamma \frac{e^{\int_\gamma \omega}l_{\gamma_o}}{\sinh\frac{l_\gamma}2}
\frac{1}{\sqrt{2\pi}\sigma}\left[e^{-\frac{(l_\gamma-T)^2}{2\sigma^2}}e^{il_\gamma R}+e^{-\frac{(l_\gamma +T)^2}{2\sigma^2}}e^{-il_\gamma R}\right]
\end{multline}

 We want to bound from below the right hand side. We hope that this bound will tell us that $\exp(\pm iTr_j)$ cannot be too small on the left hand side, giving some information
 on the imaginary part of $r_j$, for $\Re e(r_j)\sim R$. On the right,
 the idea is that
 the main contribution should come from the geodesics with $l_\gamma \sim T$. We want to choose $R$
 so as not to be bothered by the oscillatory terms $e^{\pm il_\gamma R}$.
 For that purpose,
 $R$ and $T$ will be related in the following manner~:
\begin{lem}\cite{PR94}, Lemma 3.3, \cite{JP07}. \label{l:JPTtrick}

For any $M>1$, there exists $R\in \left[M, M\exp(\exp( 5T))\right]$ such that $\cos(Rl_\gamma)\geq \frac12$ for every closed geodesic $\gamma$ with $l_\gamma \leq 5T$.
\end{lem}
In the sequel we take $M=\exp(\exp(cT))$, ($c>0$ arbitrary) to ensure that $T$ is of order $\log\log R$.
We note that this relation between $R$ and $T$ is independent of $\omega$. 
This will allow us to modify slightly our initial problem by extending it to the case where $\omega$ can depend on $R$ (or $T$).
More precisely, we want to consider the case when $\omega=\Theta(R)\omega_o$, where $\omega_o$ is fixed and $\Theta(R)\geq 1$ is allowed to
go to infinity with $R$ at a reasonable rate.

This means that we consider a slight generalization of
\eqref{e:spec}~ (the motivation should become clearer in \S \ref{p:proof3}):

\subsection{A more general problem}
We consider a spectral problem of the form
\begin{equation}\label{e:spec2}(\cP-z)u=0
\end{equation}
where $$\cP=\cP(z)=P+i\hbar\theta(\hbar)Q(z),\qquad P=-\frac{\hbar^2\lap}2$$
where $z\in \Omega=e^{i]-s_0, s_0[}]E_{\min}, E_{\max}[$, with $0<E_{\min}<\frac12<E_{\max}<+\infty$, $0<s_0<\frac\pi{4}$. We will assume that
 $ Q(z)\in \Psi DO^1 $ depends holomorphically on $z\in \Omega$, and that $\theta(\hbar)$ is some real valued function such that $\theta(\hbar)\geq 1$ and $\hbar\theta(\hbar)\Lim_{\hbar\To 0}0$. We have in mind $\theta(\hbar)=|\log(\hbar)|$. Finally, we assume that $ Q$ is formally self-adjoint for $z$ real. Again, we call $\Sigma$ the ``spectrum'' the set of $z$ for which the
 equation $(\cP(z)-z)u=0$ has a solution.

The results of \S \ref{p:semiclass} can be generalized as follows~: for any $E_{\min}<E_1\leq E_2<E_{\max}$,
\begin{equation}\label{e:Weyl3}\sharp\left\{ z\in\Sigma, E_1\leq \Re e(z)\leq E_2\right\}=\frac{1}{(2\pi\hbar)^d}\left[\int_{p_o^{-1}[E_1, E_2]}dxd\xi+\cO({\hbar\theta(\hbar)})\right].
\end{equation}
One can show that $\frac{\Im m(z)}{\hbar\theta(\hbar)}$ has to stay bounded for $z\in\Sigma$.
Taking $\theta(\hbar)=|\log(\hbar)|$, $E_1=E-c\hbar\theta(\hbar)$ and $E_2=E+c\hbar\theta(\hbar)$, one has
$$q^-_E+o(1)\leq\frac{\Im m(z)}{\hbar\theta(\hbar)}\leq q^+_E+o(1)$$
for $z\in\Sigma$ such that $E_1\leq \Re e(z)\leq E_2$.
Assuming that the geodesic flow is ergodic on $p_o^{-1}\left\{E\right\}$, {\em and for $\theta(\hbar)=|\log(\hbar)|$}, then for any $\eps>0$, any $c>0$,
\begin{equation}\label{e:LGN2}\sharp\left\{ z\in\Sigma, E-c\hbar\theta(\hbar)\leq \Re e(z)\leq
 E +c\hbar\theta(\hbar), \frac{\Im m(z)}{\hbar\theta(\hbar)}\not\in [\bar q_E-\eps, \bar q_E+\eps]\right\}=\theta(\hbar)o(\hbar^{1-d}).
\end{equation}

\begin{rem}
The paper \cite{Sj00} only treats the case $\theta(\hbar)=1$. But the method
of \cite{Sj00}, \S5 can be adapted in a straighforward way to show the following~: consider the spectral problem
\begin{equation}\label{e:spec3}(\cP-z)u=0
\end{equation}
where $$\cP=\cP(z)=P+i\hbar\theta Q(z),\qquad P=-\frac{\hbar^2\lap}2$$
where $z\in \Omega=e^{i]-s_0, s_0[}]E_{\min}, E_{\max}[$, with $0<E_{\min}<\frac12<E_{\max}<+\infty$, $0<s_0<\frac\pi{4}$.  Fix some $\eps>0$. Then
\begin{equation}\label{e:Weyl4}\sharp\left\{ z\in\Sigma, E_1\leq \Re e(z)\leq E_2\right\}=\frac{1}{(2\pi\hbar)^d}\left[\int_{p_o^{-1}[E_1, E_2]}dxd\xi+\cO(\hbar\theta)\right].
\end{equation}
uniformly in all $\theta\geq 1$ such that $\theta\hbar\leq \eps$ and $E_1, E_2$ such that $E_{\min}<E_1-2\eps,
E_2+2\eps<E_{\max}$, $|E_2-E_1|\geq\hbar\theta$.
.

For \eqref{e:LGN2} (and $\theta(\hbar)=|\log\hbar|$), the generalization of the proof in \cite{Sj00} is without surprise, but requires some rather technical changes~: we will not prove it here, but still feel allowed to ask
about Question 3 in this generalized setting. We note that to extend \eqref{e:LGN2} to more general $\theta(\hbar)$, some analyticity assumptions would be required, as in \cite{HSjV07}.
\end{rem}

We will focus our attention on the operator
$$-\hbar^2\frac{\lap_{\theta(\hbar)\omega}}2=-\hbar^2\frac{\lap}2+\hbar^2\theta(\hbar) \la \omega, d.\ra-\hbar^2\theta(\hbar)^2\frac{\norm{\omega}_x^2}2, \qquad \hbar\To 0,$$
when $\omega$ has coefficients in $\IR$.

We generalize Question 3 as\\
{\bf (Q3'')} If $\omega\not=0$, prove that there exists
a sequence  $\hbar_n\To 0$, and $z_n\in Sp(-\hbar_n^2\frac{\lap_{\theta(\hbar_n)\omega}}2)$
with $\Re e(z_n)\in [\frac12-C\hbar_n \theta(\hbar_n), \frac12+C\hbar_n\theta(\hbar_n)]$, 
such that $\frac{\Im m(z_n)}{\hbar_n
\theta(\hbar_n)}\not\To 0$.

\subsection{Heuristic discussion of the parameters $\sigma, R, T$\label{p:proof3}}
We start again from the trace formula
\eqref{e:tracegauss}, considering the case where
$\omega=\Theta(R)\omega_o$, $\Theta(R)=\theta(R^{-1})\geq 1$. Here $R^{-1}$ is going to play the role of the small parameter $\hbar$. The form $\omega_o$ is fixed, and we normalize it to have stable norm $\norm{\omega_o}_s=1$.
\begin{eqnarray}\norm{\omega}_s &=&\sup\left\{\int_{S^*M} \omega \,d\mu, \mu\in\cM_{\frac12}\right\}\\
&=&\sup_{\gamma} \frac{\int_\gamma \omega}{l_\gamma}.
\end{eqnarray}
The first line can be considered as a definition of the stable norm (valid for 
a general compact riemannian manifold $M$), whereas
the second line holds on negatively curved manifolds because of the density of the closed geodesics. Using the definition 
${\rm Pr}(\omega)=\sup\left\{h_{KS}(\mu)+\int_{S^*M} \omega \,d\mu, \mu\in\cM_{\frac12}\right\}$, it is not difficult to show that $$\lim_{t\To \infty}{\rm Pr}(t\omega)-\abs{t}\norm{\omega}_s=\sup\left\{h_{KS}(\mu), 
\mu\in\cM_{\frac12}, \int_{S^*M} \omega \,d\mu=\norm{\omega}_s\right\}.$$ On a surface, the right-hand side vanishes \cite{A03}.
Besides, for any $T$ one can find a closed geodesic $\gamma$ with $l_\gamma\in [T-1, T]$, and such that
\begin{equation}\label{e:sn}\frac{\int_\gamma \omega_o}{l_{\gamma}}\geq \norm{\omega_o}_s(1+o_T(1))= (1+o_T(1)),\end{equation}
where $o_T(1)$ goes to $0$ as $T$ approaches $+\infty$.
Simply recall that for any $0<\delta<1$,
\begin{equation}\label{e:H}\lim \frac{\log \sharp\left\{\gamma, l_\gamma\in [T-1, T],
\frac{\int_\gamma \omega_o}{l_{\gamma_o}}\geq (1-\delta)\right\}}{T}=
H(1-\delta)>0,\end{equation}
where
$$H(\alpha)=\sup\left\{h_{KS}(\mu), \mu\in\cM_{\frac12}, \int_{S^*M} \omega_o\,d\mu=\alpha\right\}.$$
The function $H$ is continuous, concave on $[-1, 1]$, real-analytic
on $]-1, 1[$ \cite{BabLed98}. And again, $H(-1)=H(1)=0$ on a compact surface \cite{A03}.

In \eqref{e:tracegauss}, we have not said yet how $\sigma$ will depend on $R$ and $T$. For the moment, let us decide {\em a priori} that $\sigma$ should be such that the term
$$\frac{Area(M)}{4\pi}\int_{-\infty}^{+\infty}r\tanh\pi r \left[e^{-\frac{\sigma^2}2(r-R)^2}e^{-iTr}+e^{-\frac{\sigma^2}2(r+R)^2}e^{iTr}\right]dr$$
is negligible compared to the sum $\sum_\gamma$.
Remember that $R$ and $T$ are chosen so as to satisfy Lemma \ref{l:JPTtrick}.
Then, fixing $1>\delta>0$, the right hand side of \eqref{e:tracegauss} should be bounded from below by
\begin{equation}\label{e:below}\sigma^{-1} e^{TH(1-\delta)-T/2}e^{\Theta(R)(1-\delta) T}e^{-\frac{1}{2\sigma^2}}.\end{equation}
We want to use the fact that this grows quite fast with $T$.
On the other hand, looking at the left hand side of \eqref{e:tracegauss},
we cannot hope to do better than to bound it from above by
\begin{equation}\label{e:above}\sharp\left\{j, |\Re e(r_j)-R|\leq \sigma^{-1} \right\} e^{\frac{\sigma^2}2 \sup_j \Im m(r_j)^2} e^{T\sup_j |\Im m(r_j)|},\end{equation}
where each time the $\sup_j$ should be restricted to the indices $j$ such that $|\Re e(r_j)-R|\leq \sigma^{-1}$. 

This heuristic argument would give an inequality
\begin{equation}\label{e:comparison}\sharp\left\{j, |\Re e(r_j)-R|\leq \sigma^{-1} \right\} e^{\frac{\sigma^2}2 \sup_j \Im m(r_j)^2} e^{T\sup_j |\Im m(r_j)|}\geq \sigma^{-1} e^{TH(1-\delta)-T/2}e^{\Theta(R)(1-\delta) T}e^{-\frac{1}{2\sigma^2}},\end{equation}
obtained by comparing the lower bound \eqref{e:below}
and the upper bound \eqref{e:above}. Again, the hope is to compare the powers of $e^T$ on both sides to prove that $\sup_j |\Im m(r_j)|$ cannot be arbitrarily small.

Consider the case $\Theta(R)=1$, which is the case we were originally interested in. If we take $\sigma$ to be a constant, then
by Weyl's law we have
$\sharp\left\{j, |\Re e(r_j)-R|\leq \sigma^{-1} \right\}\sim R\geq \exp(\exp(c T))$. In this case \eqref{e:comparison}
cannot bring any useful information. On the other hand, if we want to choose $\sigma$ such that 
$\sharp\left\{j, |\Re e(r_j)-R|\leq \sigma^{-1} \right\}$ is bounded, we are led to take $\sigma\sim R$; in this case the term $e^{\frac{\sigma^2}2 \sup_j \Im m(r_j)^2} $ will be too large to yield any interesting information.

We see that the method only has a chance to work if $T\Theta(R)\gg\log R$. From now on
we take $\Theta(R)\geq \log R$, and always such that $R^{-1}\Theta(R)\To 0$. We also take
$\sigma^{-2}=C\Theta(R)$ with $C$ large. We must note that the parameters $r_j$ correspond
to the eigenvalues of $-\lap_{\Theta(R)\omega}$, and thus they also depend on $R$.

\subsection{Proof of Theorem \ref{t:q3}}The right hand side of \eqref{e:tracegauss} is easy to understand. The term 
\begin{equation}\label{e:int}\int_{-\infty}^{+\infty}r\tanh\pi r \left[e^{-\frac{\sigma^2}2(r-R)^2}e^{-iTr}+e^{-\frac{\sigma^2}2(r+R)^2}e^{iTr}\right]dr\end{equation} is $\cO(\sigma^{-1}R)$, whereas the $\sum_\gamma$
has modulus greater than
\begin{multline}
\frac12\sum_{l_\gamma\leq 5 T} \frac{e^{\int_\gamma \omega}l_{\gamma_o}}{\sinh\frac{l_\gamma}2}
\frac{1}{\sqrt{2\pi}\sigma}\left[e^{-\frac{(l_\gamma-T)^2}{2\sigma^2}}+e^{-\frac{(l_\gamma +T)^2}{2\sigma^2}}\right]
+ \sum_{l_\gamma\geq 5T} \frac{e^{\int_\gamma \omega}l_{\gamma_o}}{\sinh\frac{l_\gamma}2}
\frac{1}{\sqrt{2\pi}\sigma}\left[e^{-\frac{(l_\gamma-T)^2}{2\sigma^2}}+e^{-\frac{(l_\gamma +T)^2}{2\sigma^2}}\right]\cos(l_\gamma R)
\\ \geq 
\frac12\sum_{T-1\leq l_\gamma\leq  T} \frac{e^{\int_\gamma \omega}l_{\gamma_o}}{\sinh\frac{l_\gamma}2}
\frac{1}{\sqrt{2\pi}\sigma}\left[e^{-\frac{(l_\gamma-T)^2}{2\sigma^2}}+e^{-\frac{(l_\gamma +T)^2}{2\sigma^2}}\right]
+ \sum_{l_\gamma\geq 5T} \frac{e^{\int_\gamma \omega}l_{\gamma_o}}{\sinh\frac{l_\gamma}2}
\frac{1}{\sqrt{2\pi}\sigma}\left[e^{-\frac{(l_\gamma-T)^2}{2\sigma^2}}+e^{-\frac{(l_\gamma +T)^2}{2\sigma^2}}\right]\cos(l_\gamma R)
\\ \geq 
\frac{1}{\sqrt{8\pi}\sigma} \left[e^{TH(1-\delta)-T/2}e^{\Theta(R)(1-\delta) T}
e^{-\frac{1}{2\sigma^2}} +\cO(1)\right],\label{e:modulus}
\end{multline}
(using \eqref{e:H}), and thus is much greater than the integral \eqref{e:int}. To get the last $\cO(1)$ we have used the exponential growth of the number of closed geodesics. The left hand side of \eqref{e:tracegauss}
is more complicated to bound from above, since the $r_j$ now depend on $R$.

\begin{prop} \label{p:f}Take $\Theta(R)\geq \log R$, and such that $R^{-1}\Theta(R)\To 0$.
Take $\sigma^{-2}=C\Theta(R)$. Let $f(R)$ be such that $\sigma^2f(R)^2\gg T\Theta(R)$.
Then
$$|\sum_j e^{-\frac{\sigma^2}2(r_j-R)^2}e^{-iTr_j}|
\leq  \sharp\left\{j, |\Re e(r_j)-R|\leq f(R) \right\} e^{\frac{\sigma^2}2 \sup_j \Im m(r_j)^2} e^{T\sup_j |\Im m(r_j)|}+\cO(1),$$
where the $\sup_j$ are taken over the set of indices $j$ such that  $|\Re e(r_j)-R|\leq f(R) $.
\end{prop}

\begin{proof}
\begin{lem} We have an a priori bound $|\Im m(r_j)|\leq c\,\Theta(R)$, where $c$ depends only on $\omega_o$.
\end{lem}

\vspace{.3cm}
Indeed, let $r_j=x+iy$ and $f\in L^2(M)$ be such that $\norm{f}_{L^2}=1$ and
$$-\lap f+2\Theta(R)\la \omega_o, df\ra-\Theta(R)^2 \norm{\omega_o}^2_x f=\left(\frac14+r_j^2\right)f.$$
Taking the scalar product with $f$, we get
\begin{equation}\label{e:1}\frac14+x^2-y^2=\int |\nabla f|^2 -\Theta(R)^2\int 
\norm{\omega_o}_x^2\abs{f}^2\end{equation}
and
\begin{equation}\label{e:2}2xy=-2\Theta(R) i\int\la\omega_o, df\ra\bar f . \end{equation}
Equation \eqref{e:2} yields $\abs{xy}\leq c\, \Theta(R) \sqrt{\int |\nabla f|^2}.$ Equation \eqref{e:1}
implies that $x^2\geq \int |\nabla f|^2 -c^2\Theta(R)^2.$ If $\abs{x}\geq \frac12\sqrt{\int |\nabla f|^2}$ then we are done, by \eqref{e:2}. If $\abs{x}\leq \frac12\sqrt{\int |\nabla f|^2}$, then \eqref{e:1} implies that $\int |\nabla f|^2 \leq 2c^2 \Theta(R)^2$ and that $y^2\leq \frac14 +5c^2 \Theta(R)^2.$
The lemma follows.

\vspace{.4cm}
We now break the sum
$\sum_j e^{-\frac{\sigma^2}2(r_j-R)^2}e^{-iTr_j}$ into three parts~: $I=\sum_{j,
\Re e(r_j)\leq R- f(R)}$,  $II=\sum_{j,|\Re e(r_j)-R|\leq f(R)}$ and 
$III=\sum_{j,
\Re e(r_j)\geq R+ f(R)}$.

The last sum $III$ is bounded by
$$e^{\frac{\sigma^2}2 c^2\Theta(R)^2} e^{cT\Theta(R)}\sum_{j,
\Re e(r_j)\geq R+ f(R)}e^{-\frac{\sigma^2}2(\Re e(r_j)-R)^2} .$$
We decompose this sum into $\sum_{n\geq 0} \sum_{R+ f(R)+n\leq\Re e(r_j)\leq R+ f(R)+n+1}$, and by Weyl's law in the form \eqref{e:Weyl4}, this is dominated by
$$
e^{cT\Theta(R)}\sum_{n\geq 0} (R+n+f(R))\Theta(R) e^{-\frac{\sigma^2}2(n+f(R))^2} \leq
e^{cT\Theta(R)}\Theta(R)\int_{f(R)-1}^{+\infty}(R+x)e^{-\frac{\sigma^2 x^2}2}dx
$$
and with our relations between $T, R, f(R)$ and $\sigma$, this last quantity is $\cO(1)$.

Concerning the first sum $I$, we bound it by
\begin{equation}\label{e:firstsum}e^{\frac{\sigma^2}2 c^2\Theta(R)^2} e^{cT\Theta(R)}\sum_{j,
\Re e(r_j)\leq R- f(R)}e^{-\frac{\sigma^2}2(\Re e(r_j)-R)^2}.\end{equation} The subsum $\sum_{j,
\Re e(r_j)\leq -R+ f(R)}$ can be treated as above and shown to be $\cO(1)$ (using Weyl's law in the form \eqref{e:Weyl4}), and we only need to concentrate on
$\sum_{j,
|\Re e(r_j)|\leq R- f(R)}$. To bound this sum, we first need a control the number of terms.

\begin{lem}$\sharp\left\{j, |\Re e(r_j)|\leq R\right\}=\cO(R^2\Theta(R)).$\label{l:Weyl5}
\end{lem}
To that end, we use again the trace formula and write~:
\begin{equation*}
\sum_j e^{-\frac{r_j^2}{2R^2}} 
=\frac{Area(M)}{4\pi}\int_{-\infty}^{+\infty}r\tanh(\pi r). e^{-\frac{r^2}{2R^2}} dr +
\sum_\gamma \frac{e^{\Theta(R)\int_\gamma \omega_o}l_{\gamma_o}}{\sinh\frac{l_\gamma}2}
\frac{R}{\sqrt{2\pi}}e^{-R^2 l_\gamma^2/2}.
\end{equation*}
On the right, the term $\sum_\gamma$ is clearly $o(1)$ whereas the $\int$ term is  of order $ R^2$.
On the left, we break the sum into $\sum_{j, |\Re e (r_j)|\leq R}$ and $\sum_{j, |\Re e (r_j)|\geq R}$. As above, we can use Weyl's law in the form \eqref{e:Weyl4} to show that the $\sum_{j, |\Re e (r_j)|\geq R}e^{-\frac{r_j^2}{2R^2}} $
is $\cO \sum_{n\geq 0}(R+n+1)\Theta(R)e^{-\frac{(R+n)^2}{2R^2}}=\cO(R^2\Theta(R))$. Thus we have
$$\sum_{j, |\Re e( r_j)|\leq R}e^{-\frac{\Re e (r_j)^2-\Im m(r_j)^2}{2R^2}}e^{-\frac{i \Re e(r_j)\Im m(r_j)}{R^2}}
=\cO(R^2\Theta(R)).$$
Remember that $|\Im m(r_j)|\leq c\,\Theta(R)$ and that $R^{-1}\Theta(R)\To 0$. Thus, for $|\Re e (r_j)|\leq R$ we can write $e^{-\frac{i \Re e(r_j)\Im m(r_j)}{R^2}}=1+\cO(R^{-1}\Theta(R))$. This yields
\begin{eqnarray*} e^{-1/2} 
\sharp\left\{j, |\Re e(r_j)|\leq R\right\}\left(1+\cO(R^{-1}\Theta(R))\right)&\leq& \sum_{j, |\Re e r_j|\leq R}e^{-\frac{\Re e (r_j)^2-\Im m(r_j)^2}{2R^2}}\left(1+\cO(R^{-1}\Theta(R))\right)
\\ &=&\cO(R^2\Theta(R)) 
\end{eqnarray*}
and finishes the proof of Lemma \ref{l:Weyl5}.
 
Now, we go back to the sum $\sum_{j,
|\Re e(r_j)|\leq R- f(R)}$ in \eqref{e:firstsum}, and we see that it is bounded by
$$ R^2\Theta(R) e^{\frac{\sigma^2}2 c^2\Theta(R)^2} e^{cT\Theta(R)}e^{-\sigma^2 f(R)^2/2}=\cO(1).$$
This ends the proof of Proposition \ref{p:f}
\end{proof}
We can now come back to \eqref{e:tracegauss}. Noting that $\sigma^2 \sup_j \Im m(r_j)^2=\cO(1)$,
Proposition \ref{p:f} shows that the left-hand side of the trace formula \eqref{e:tracegauss}
is bounded from above by 
 $$C \sharp\left\{j, |\Re e(r_j)-R|\leq f(R) \right\} e^{T\sup_j |\Im m(r_j)|}+\cO(1)
 \leq CRf(R)\Theta(R)e^{T\sup_j |\Im m(r_j)|}$$
 where the $\sup$ is taken over all $j$ such that $|\Re e(r_j)-R|\leq f(R)$, and where we have used again \eqref{e:Weyl4}. We can of course assume, without loss of generality, that $f(R)=\cO(R)$.
 
 On the right hand side of \eqref{e:tracegauss}, the $\int$ is $\cO(R\Theta(R)^{1/2})$, and the $\sum_\gamma$ is bounded from below by
 $$\frac{1}{\sqrt{8\pi}\sigma} e^{TH(1-\delta)-T/2}e^{\Theta(R)(1-\delta) T}
e^{-\frac{1}{2\sigma^2}} ,$$
 as in \eqref{e:modulus}. Writing
  $$CRf(R)\Theta(R) e^{T\sup_j |\Im m(r_j)|}\geq \sigma^{-1} e^{TH(1-\delta)-T/2}e^{\Theta(R)(1-\delta) T}e^{-\frac{1}{2\sigma^2}},$$
remembering that $\Theta(R)\geq \log(R)$, $\Theta(R)=o(R)$, $\sigma^{-2}=c\,\Theta(R)$
and $T\asymp\log\log R$, we see that necessarily
$$\sup_j |\Im m(r_j)|\geq (1-2\delta)\Theta(R).$$
This finishes the proof of Theorem \ref{t:q3}.  

\section{The arithmetic case.\label{s:arithm}}
Let $p\geq 3$ be a prime, $p\equiv 1\,({\rm mod } \;4)$, and $A\geq 1$ be a quadratic non-residue modulo $p$. We set
$$\Gamma=\Gamma(A, p)=\left\{\left( \begin{array}{cc} 

y_0+y_1\sqrt{A} & y_2\sqrt{p}+y_3\sqrt{Ap} \\
 y_2\sqrt{p}-y_3\sqrt{Ap} & y_0-y_1\sqrt{A}  \\
 
 \end{array}  \right), y_0, y_1, y_2, y_3\in\IZ\right\}.$$
It is a discrete cocompact subgroup of $SL(2, \IR)$ which contains only hyperbolic transformations \cite{Hej}.
We consider the hyperbolic surface $M=\Gamma\backslash \IH$.
In $M$, the lengths of the closed geodesics are the $\log x_m$, where
$$x_m=2m^2-1+2m\sqrt{m^2-1},\qquad m\in\IN.$$ We define
$$\mu(m)=\sum_{\gamma,\, l_\gamma=\log x_m}e^{\int_\gamma \omega}l_{\gamma_o}.$$
We now follow very closely the approach of \cite{Hej}, pp. 304--314. We introduce an even function $k$ on $\IR$, whose Fourier transform is nonnegative and compactly supported in $[-1, 1]$; we also assume
that $\hat k\geq 1$ on $[-\frac12, \frac12]$. We define
$$K_{\alpha}(r)=k(r)[e^{i\alpha r}+e^{-i\alpha r}].$$
 
We write again the trace formula~: for all $t>0$,
\begin{multline}\label{e:traceshift}
\sum_j K_\alpha(r_j-t)+K_\alpha(r_j+t)=\frac{Area(M)}{4\pi}\int r \tanh(\pi r) [K_\alpha(r-t)+K_\alpha(r+t)]dr\\
+2\sum_\gamma \frac{e^{\int_\gamma \omega}l_{\gamma_o}}{\sinh\frac{l_\gamma}2}\hat K_\alpha(l_\gamma)\cos(tl_\gamma).
\end{multline}
We will bound from below the right-hand side (averaged in $t$) to obtain information on the left-hand side. Denote
$$S_\alpha(t)=\sum_\gamma \frac{e^{\int_\gamma \omega}l_{\gamma_o}}{\sinh\frac{l_\gamma}2}\hat K_\alpha(l_\gamma)\cos(tl_\gamma)=
2\sum_{e^{\alpha-1}\leq x_m\leq e^{\alpha+1}}\frac{\mu(m)}{x_m^{1/2}+x_m^{-1/2}}\hat K_\alpha(\log x_m)\cos(t\log x_m).$$
\begin{prop}Let $\alpha=2\beta\log T-C$, $0<\beta\leq 1$ and $C$ large enough. Then,
$$\int_{2T-T^\beta}^{2T+T^\beta}\left(1-\frac{\abs{t-2T}}{T^\beta}\right)|S_\alpha(t)|^2 dt\geq \tilde C T^{\beta(4{\rm Pr}(\omega)-1)}.$$
\end{prop}
Although we are really interested in the quantity $\int_{2T-T^\beta}^{2T+T^\beta}|S_\alpha(t)|^2 dt$,
the reason for introducing the regularizing factor $\left(1-\frac{\abs{t-2T}}{T^\beta} \right)$ is exactly the same
as in \cite{Hej}, p. 315.
\begin{proof}
Introduce the notations 
$$\eta(m)=\frac{\mu(m)}{x_m^{1/2}+x_m^{-1/2}}\hat K_\alpha(\log x_m),$$
$$\nu(m)=\mu(m)\hat K_\alpha(\log x_m).$$
 Divide the integral $I=\int_{2T-T^\beta}^{2T+T^\beta}\left(1-\frac{\abs{t-2T}}{T^\beta}\right)|S_\alpha(t)|^2 dt$ into $I=I_1+I_2$,
where
$$I_1=\sum_{e^{\alpha-1}\leq x_m\leq e^{\alpha+1}}\eta(m)^2\int_{2T-T^\beta}^{2T+T^\beta}\left(1-\frac{\abs{t-2T}}{T^\beta}\right)\cos^2(t\log x_m)dt$$
and
$$I_1=2\sum_{e^{\alpha-1}\leq x_k< x_m\leq e^{\alpha+1}}\eta(m)\eta(k)\int_{2T-T^\beta}^{2T+T^\beta}\left(1-\frac{\abs{t-2T}}{T^\beta}\right)\cos(t\log x_m)\cos(t\log x_k)dt.$$
The idea is that the $x_m, x_k$, with $x_m\not= x_k$, are well-spaced, implying that the oscillatory integral $I_2$ is small compared to $I_1$.

\begin{lem}  For $T\geq 1$ and $\lambda\in\IR$,
$$\int_{2T-T^\beta}^{2T+T^\beta}\left(1-\frac{\abs{t-2T}}{T^\beta}\right)e^{i\lambda t}dt=\cO\left[\min\left(T^\beta, \frac{1}{\lambda^2T^\beta}\right)\right].$$
\end{lem}

Let us first consider $I_1$.
\begin{eqnarray*}\int_{2T-T^\beta}^{2T+T^\beta}\left(1-\frac{\abs{t-2T}}{T^\beta}\right)\cos^2(t\log x_m)dt&=&\int_{2T-T^\beta}^{2T+T^\beta}\left(1-\frac{\abs{t-2T}}{T^\beta}\right)\frac{1+\cos(2t\log x_m)}2dt\\
&=& \frac{T^\beta}2+\cO(T^{-\beta}\log x_m^{-2}).
\end{eqnarray*}
Hence, 
\begin{eqnarray}I_1&\geq& \left(\frac{T^\beta}2+o(1)\right)\sum_{e^{\alpha-1}\leq x_m\leq e^{\alpha+1}}\eta(m)^2
\\& \geq&\left(\frac{T^\beta}2+o(1)\right)\sum_{e^{\alpha-1}\leq x_m\leq e^{\alpha+1}}\frac{\nu(m)^2}{(x_m^{1/2}+x_m^{-1/2})^2}
\\& \geq& c_1T^\beta e^{-\alpha} \sum_{e^{\alpha-1}\leq x_m\leq e^{\alpha+1}}\nu(m)^2.\label{e:I_1}
\end{eqnarray}
The right-hand side of \eqref{e:I_1} will be estimated later.
We now turn to $I_2$, and want to show that it is much smaller than $I_1$. The integral
$$\int_{2T-T^\beta}^{2T+T^\beta}\left(1-\frac{\abs{t-2T}}{T^\beta}\right)\cos(t\log x_m)\cos(t\log x_k)dt$$
is smaller than $\frac{1}{T^\beta(\log x_m-\log x_k)^2}$. We can ensure that
$$\frac{1}{\abs{\log x_m-\log x_k}}\leq cT^\beta$$
for $e^{\alpha-1}<x_k<x_m<e^{\alpha+1}$, by choosing $\alpha$ in an appropriate range~: writing
$$\log x_m-\log x_{m-1}\sim \frac{x_m-x_{m-1}}{x_m}\sim \frac2m\sim \frac4{\sqrt{x_m}},$$
we see we have to take $\alpha\leq 2\beta\log T-C$ ($C$ large). More generally, we have by the intermediate value theorem
$$|\log x_m-\log x_{k}|\geq\tilde C e^{-\alpha/2}|m-k| .$$
The analysis
done by \cite{Hej}, pp. 310--311, can be applied {\em verbatim} to show that
\begin{equation}\label{e:I_2}\abs{I_2}\leq\frac{\tilde C}{T^\beta}\sum_{e^{\alpha-1}\leq x_m\leq e^{\alpha+1}}\nu(m)^2.\end{equation}
Comparing \eqref{e:I_1} and \eqref{e:I_2}, we see that
$$\abs{I_2}\leq \frac{I_1}{100}$$
provided $\alpha\leq 2\beta\log T-C$ with $C$ sufficiently large. Thus,
$$I\geq \frac{99}{100}I_1.$$
To complete our estimate for $I$, we must return to equation \eqref{e:I_1}. Clearly,
$$\sum_{e^{\alpha-1}\leq x_m\leq e^{\alpha+1}}\nu(m)^2\geq 
\sum_{e^{\alpha-1/2}\leq x_m\leq e^{\alpha+1/2}}\mu(m)^2.$$
We write the Cauchy-Schwarz inequality,
$$\left[\sum_{e^{\alpha-1/2}\leq x_m\leq e^{\alpha+1/2}}\mu(m)\right]^2\leq
\left(\sum_{e^{\alpha-1/2}\leq x_m\leq e^{\alpha+1/2}} 1\right)
\left(\sum_{e^{\alpha-1/2}\leq x_m\leq e^{\alpha+1/2}}\mu(m)^2\right).$$
But
$$\sum_{e^{\alpha-1/2}\leq x_m\leq e^{\alpha+1/2}}\mu(m)= \sum_{\gamma, \alpha-1/2\leq l_\gamma\leq\alpha+1/2}e^{\int_\gamma \omega}l_{\gamma_o}\geq \tilde C\,e^{\alpha {\rm Pr}(\omega)},$$
see \cite{PP}, p. 117. 
On the other hand,
$$\sum_{e^{\alpha-1/2}\leq x_m\leq e^{\alpha+1/2}} 1=\cO(e^{\alpha/2}).$$
We obtain this way
$$\sum_{e^{\alpha-1/2}\leq x_m\leq e^{\alpha+1/2}}\mu(m)^2\geq C e^{2\alpha {\rm Pr}(\omega)-\alpha/2}.$$
We have proved
$$\int_{2T-T^\beta}^{2T+T^\beta}\left(1-\frac{\abs{t-2T}}{T^\beta}\right)|S_\alpha(t)|^2 dt\geq \tilde C T^{\beta(4{\rm Pr}(\omega)-1)}.$$
\end{proof}
This implies
$$S_\alpha(t)\geq C  t^{\beta(2{\rm Pr}(\omega)-1)}$$
for some $t\in [2T-T^\beta, 2T+T^\beta]$.

Now consider the integral $\int r \tanh(\pi r) K_\alpha(r-t)dr$ or $\int r \tanh(\pi r) K_\alpha(r+t)dr$
in \eqref{e:traceshift}. We write
$$\int r \tanh(\pi r) K_\alpha(r-t)dr =\int r \tanh(\pi r)k(r-t) [e^{i\alpha (r-t)}+e^{-i\alpha (r-t)}]dr.$$
To evaluate 
$\int (r+t) \tanh(\pi (r+t))k(r) e^{i\alpha r} dr$, we shift the integral over $\IR$ to an integral over
$(\frac12-\eps)i +\IR$, and we find that the integral is $\cO(te^{-|\alpha|(\frac12-\eps)})$ for any $\eps>0$.

\begin{lem}
$\int r \tanh(\pi r) K_\alpha(r-t)dr$ or $\int r \tanh(\pi r) K_\alpha(r+t)dr=\cO(te^{-|\alpha|(\frac12-\eps)})=\cO(t^{1+\eps-\beta})$
for any $\eps>0$.
\end{lem}

We finally turn to 
$\sum_j K_\alpha(r_j-t)+K_\alpha(r_j+t)$. Fixing a small $\eps>0$, one sees using Weyl's law, the fact that $k$ is rapidly decreasing in any horizontal strip -- and the fact that the $\Im m(r_j)$ are bounded -- that
$$\sum_j K_\alpha(r_j-t)=\sum_{j, |\Re e(r_j)-t|\leq t^\eps}K_\alpha(r_j-t)+\cO(t^{-\infty}).$$
Similarly,
$$\sum_j K_\alpha(r_j+t)=\sum_{j, |\Re e(r_j)+t|\leq t^\eps}K_\alpha(r_j+t)+\cO(t^{-\infty}).$$
We see that
$$|\sum_j K_\alpha(r_j-t)+K_\alpha(r_j+t)|\leq C e^{\alpha\sup_j|\Im m(r_j)|}t^{1+\eps}\leq C t^{2\beta\sup_j|\Im m(r_j)|}t^{1+\eps}$$
where the $\sup_j$ is taken over the $j$ such that $|\Re e(r_j)\pm t|\leq t^\eps.$

We have proved that there exists $t\in [T-T^\beta, T+T^\beta]$, and $r_j$ with $|\Re e(r_j)- t|\leq t^\eps,$
such that
$$t^{2\beta\sup_j|\Im m(r_j)|}t^{1+\eps}\geq \tilde C t^{\beta(2{\rm Pr}(\omega)-1)}.$$
In particular, if $T$ is large enough, this implies
$$\sup\left\{|\Im m(r_j)|, |\Re e(r_j)\pm T|\leq T^\beta\right\}\geq {\rm Pr}(\omega)-\frac12 -\frac{1+\eps}{2\beta},$$
and this proves Theorem \ref{t:q3arithm}.

\section{Symbol classes\label{s:symbols}}

Following \cite{DS}, for any $0\leq \delta <1/2$ we introduce the symbol class 
\bequ\label{e:symbol-eps}
S_\delta^{m}\defeq\set{a\in C^\infty(T^*M),\ |\partial_x^\alpha \partial_{\xi}^\beta a|\leq 
C_{\alpha,\beta}\, \hbar^{-\delta|\alpha+\beta|}\,\la\xi\ra^{m-|\beta|}}\,.
\eequ
If $\delta=0$ we will just denote $S^{m}$.
We will denote $\Op_\hbar(a(x, \xi))=\Op(a(x, \hbar\xi))$, where $\Op$ is a quantization procedure on $M$.
The quantization of any $a\in S_\delta^{0}$ leads
to a bounded operator on $L^2(M)$ (the norm being bounded uniformly in $\hbar$), see \cite{DS}. 
We will denote $\Psi DO_\delta^m=\Op_\hbar(S_\delta^{m}).$

We use~:

\begin{prop}\label{p:PDO}Let $a\in S_\delta^{m}, b\in S_{\delta'}^{n}$ with $0\leq\delta'\leq\delta<1/2$.
Then\\
(i) $\Op_\hbar(a)\Op_\hbar(b)-\Op_\hbar(ab)\in \hbar^{1-\delta-\delta'}\Op_\hbar(S_\delta^{m+n-1})$.\\
(ii) $[\Op_\hbar(a),\Op_\hbar(b)]-\frac\hbar{i}\Op_\hbar(\left\{a, b\right\})\in \hbar^{2(1-\delta-\delta')}\Op_\hbar(S_\delta^{m+n-2})$.\\
\end{prop}

We also use a local form of the Calderon-Vaillancourt estimate \cite{DS}~:
\begin{prop}\label{p:CV} There exists $K\in\IN$ depending only on the dimension of $M$, such that the following holds. Take $A=\Op_\hbar(a)$ where $a\in\Psi DO_\delta^2$. Let $I$ be an open interval of $\IR^+$ and let $\lambda$ belong to $I$. Then, there exists $C>0$, and $C(a,\lambda)$ depending on a finite number of seminorms of $a$ (uniform in $\lambda$ if it stays inside a compact subset of $I$), such that, for all $u\in L^2(M)$,
$$\norm{Au}_{L^2}\leq C \left( \sup_{p_o^{-1}(I)}|a|+\sum_{k=1}^K\hbar^k\sup_{p_o^{-1}(I)}|D^{2k} a|  \right)\norm{u}_{L^2}+C(a, \lambda) \norm{(P-\lambda)u}_{L^2}.$$
In fact $C(a, \lambda)$ is controlled by the supremum norm of $\frac{a}{p_o-\lambda}$ and a finite number of its derivatives outside $p_o^{-1}(I)$.
Similarly we have
$$\abs{\la u,Au\ra}\leq C \left( \sup_{p_o^{-1}(I)}|a|+\sum_{k=1}^K\hbar^k\sup_{p_o^{-1}(I)}|D^{2k} a|  \right)\norm{u}_{L^2}+C(a, \lambda) \norm{(P-\lambda)u}^2_{L^2}.$$
\end{prop}

\section{Appendix~: Sj\"ostrand's proof of Proposition \ref{p:plagiat}\label{a:plagiat}}
In order to prove Proposition \ref{p:plagiat}, we first want to bound the norm and trace norm of
$$f\left(\frac{2P-1}\hbar\right)(\tQ_T-Q_T)f\left(\frac{2P-1}\hbar\right).$$
We write a Calderon-Vaillancourt type estimate,
$$\norm{(\tQ_T-Q_T)u}\leq \left(\sup_{p_o^{-1}]\frac12-\eps, \frac12+\eps[}(\tilde q^T-\tilde q^T)+\cO(\hbar^{1-2\delta})\right)\norm{u}+\cO(1)\norm{(2P-1)u},$$
where $\delta$ is as in Proposition \ref{p:ave}.
Besides, $\norm{(2P-1)f\left(\frac{2P-1}\hbar\right)}=\cO(\hbar)$. It follows that
$$\left\Vert f\left(\frac{2P-1}\hbar\right)(\tQ_T-Q_T)f\left(\frac{2P-1}\hbar\right)\right\Vert\leq
\norm{f}^2_\infty \left(\sup_{p_o^{-1}]\frac12-\eps, \frac12+\eps[}(\tilde q^T-\tilde q^T)+\cO(\hbar^{1-2\delta})\right).$$
For the trace class norm, we need to be even more careful than in \cite{Sj00}. Instead of using the G\aa rding inequality, we use the existence of a positive quantization. If we choose such, we have directly that
$f\left(\frac{2P-1}\hbar\right)(\tQ_T-Q_T)f\left(\frac{2P-1}\hbar\right)\geq 0$ in the operator sense. Thus,
\begin{eqnarray*}\left\Vert f\left(\frac{2P-1}\hbar\right)(\tQ_T-Q_T)f\left(\frac{2P-1}\hbar\right)\right\Vert_{1}&=&
\Tr\, f\left(\frac{2P-1}\hbar\right)(\tQ_T-Q_T)f\left(\frac{2P-1}\hbar\right)\\
&=& \Tr\, f\left(\frac{2P-1}\hbar\right)^2(\tQ_T-Q_T) \\
&=&\Tr\, \frac1{2\pi}\int \hat{f^2}(t)e^{it\frac{2P-1}\hbar}(\tQ_T-Q_T)dt.
\end{eqnarray*}
Writing the expansion of $e^{it\frac{2P-1}\hbar}$ as a Fourier Integral Operator, writing the trace as the integral of the kernel, and applying the stationary phase method in the time-energy variables,
we obtain an asymptotic expansion
\begin{multline*}\Tr\, \frac1{2\pi}\int \hat{f^2}(t)e^{it\frac{2P-1}\hbar}(\tQ_T-Q_T)dt\\
= C_d\hbar^{2-d}\left[\hat{f^2}(0)\int_{p_o^{-1}\left(\frac12\right)}(q^T-\tilde q^T)L_{\frac12}(d\rho)+\sum_{k=1}^{N-1}\hbar^k D^{2k}_t\hat{f^2}(0)\int_{p_o^{-1}\left(\frac12\right)}D^{2k}_\rho(q^T-\tilde q^T)L_{\frac12}(d\rho)+\cO(\hbar^{N(1-2\delta)})\right],
\end{multline*}
where $D^{2k}_t$ and $D^{2k}_\rho$ are differential operators of degree $\leq 2k$, respectively on
$\IR$ and $T^*M$. Note that the term $\hbar^k D^{2k}_t\hat{f^2}(0)\int_{p_o^{-1}\left(\frac12\right)}D^{2k}_\rho(q^T-\tilde q^T)L_{\frac12}(d\rho)$ is a $\cO(\hbar^{k(1-2\delta)})L_{\frac12}(\tilde q^T\not=q^T)=o(1)L_{\frac12}(\tilde q^T\not=q^T).$ This proves, in particular, Corollary \ref{c:maincoro}.

To finish the proof of Proposition \ref{p:plagiat}, there remains to study the invertibility of $z-\tilde\cP_T$. Recall the identity
$$\norm{(A+iB)u}^2=\norm{Au}^2+\norm{Bu}^2+i\la u, [A, B] u\ra,$$
if $A, B$ are bounded self-adjoint operators. Thus,
\begin{multline*}
2\norm{(\tilde\cP_T-z)u}^2\geq\norm{(P+i\hbar\hat Q_T-z)u}^2-\cO(\hbar^{4(1-\delta)})(\norm{(2P-1)u}^2+\norm{u}^2)\\
\geq\norm{(P-\Re e(z))u}^2+\hbar^2\norm{\left(\frac{\Im m(z)}\hbar-\hat Q_T   \right)u}^2
+i\hbar\la u, [P, \hat Q_T]u\ra \\-
\cO(\hbar^{4(1-\delta)})(\norm{(2P-1)u}^2+\norm{u}^2)\\
=\norm{(P-\Re e(z))u}^2+\hbar^2\norm{\left(\frac{\Im m(z)}\hbar-\hat Q_T   \right)u}^2\\
+\left(\cO(1)\frac{\hbar^2}T(1+\norm{f}^2_\infty)+\cO(\hbar^{3-2\delta})\right)\norm{u}^2
+\cO(\hbar^2)\norm{(P-\Re e(z))u}^2
\end{multline*}
We have used \eqref{e:comm} (or Proposition \ref{p:CV}), and the same for $\tilde Q_T$. 
We find that
\begin{equation}\label{e:real}
\sqrt{3}\norm{(\tilde\cP_T-z)u}\geq \norm{(P-\Re e(z))u}-(\cO(\frac\hbar{\sqrt T})+\cO(\hbar^{\frac32-\delta}))\norm{u}.
\end{equation}
On the other hand, we have
\begin{multline}
\Im m \left\la \frac1{\hbar}(z-\tilde\cP_T)u, u\right\ra\\
= \left\la \left(\frac{\Im m(z)}{\hbar}-\hat Q_T\right)u, u\right\ra +\cO(\hbar^{1-2\delta})(\norm{u}+\norm{(\Re e(z)-P)u})\norm{u}\\
 \left\la\left( \frac{\Im m(z)}{\hbar}- Q_T+f\left(\frac{2P-1}\hbar\right)(\tQ_T-Q_T)f\left(\frac{2P-1}\hbar\right)\right)u, u\right\ra \\+\cO(\hbar^{1-2\delta})(\norm{u}+\norm{(\Re e(z)-P)u})\norm{u}\\
= \left\la\left( \frac{\Im m(z)}{\hbar}- Q_T+f\left(\frac{2\Re e(z)-1}\hbar\right)^2(\tQ_T-Q_T)\right)u, u\right\ra \\+\cO(\hbar^{1-2\delta})(\norm{u}+\norm{(\Re e(z)-P)u})\norm{u}
 \\-\left(2\sup_{p_o^{-1}]\frac12-\eps, \frac12+\eps[}(q^T-\tilde q^T)\norm{f}_\infty\norm{f'}_\infty
 +\cO(\hbar^{1-2\delta})\right)\norm{u}\left\Vert \frac{P-\Re e(z)}\hbar u  \right\Vert
 \label{e:im}
\end{multline}
by the same trick as in \cite{Sj00}, (3.19). Recall that we are interested in a region where
$z-\frac12=\cO(\hbar).$

\vspace{.4cm}
Let $\alpha(E)>0$ be a continuous function defined on a bounded interval $J$ containing $0$, and restrict $z$ by assuming that
$$\frac{\Im m(\zeta)}{2\hbar}-q^T+f\left(\frac{\Re e(\zeta)}\hbar\right)^2(q^T-\tilde q^T)\geq \alpha\left(\frac{\Re e(\zeta)}\hbar\right), $$ near $p_o^{-1}\left(\frac12\right), \frac{\Re e(\zeta)}\hbar\in J$
(where $\zeta=2z-1$). It follows from the G\aa rding inequality that for such $z$
\begin{multline*}\left\la\left( \frac{\Im m(z)}{\hbar}- Q_T+f\left(\frac{2\Re e(z)-1}\hbar\right)^2(\tQ_T-Q_T)\right)u, u\right\ra\\
\geq \left( \alpha\left(\frac{\Re e(\zeta)}\hbar\right)-\cO(\hbar^{1-2\delta})\right)\norm{u}^2-\cO(1)
\norm{u}\norm{(P-\Re e(z))u}.
\end{multline*}
Using this in \eqref{e:im}, we get
\begin{multline}
\Im m \left\la \frac1{\hbar}(z-\tilde\cP_T)u, u\right\ra \geq \left( \alpha\left(\frac{\Re e(\zeta)}\hbar\right)-\cO(\hbar^{1-2\delta})\right)\norm{u}^2\\
 -\left(2\sup_{p_o^{-1}]\frac12-\eps, \frac12+\eps[}(q^T-\tilde q^T)\norm{f}_\infty\norm{f'}_\infty
 +\cO(\hbar^{1-2\delta})\right)\norm{u}\left\Vert \frac{P-\Re e(z)}\hbar u  \right\Vert.
\end{multline}
Reasoning as in \cite{Sj00}, (3.25) and (3.26), we find finally
\begin{multline*}
\Im m \left\la \frac1{\hbar}(z-\tilde\cP_T)u, u\right\ra \geq \left( \alpha\left(\frac{\Re e(\zeta)}\hbar\right)-\cO(\hbar^{1-2\delta})\right)\norm{u}^2\\
 -\sqrt 3\left(2\sup_{p_o^{-1}]\frac12-\eps, \frac12+\eps[}(q^T-\tilde q^T)\norm{f}_\infty\norm{f'}_\infty
 +\cO(\hbar^{1-2\delta})\right)\norm{u}\left\Vert \frac{\tilde\cP_T-z}\hbar u  \right\Vert
 \\-(\cO(\frac1{\sqrt T})+\cO(\hbar^{\frac12-\delta}))\norm{u}^2 .
\end{multline*}
and 
\begin{multline*}
\left( \alpha\left(\frac{\Re e(\zeta)}\hbar\right)+\cO(\frac1{\sqrt T})+\cO(\hbar^{\frac12-\delta})\right)\norm{u}
\\\leq \left[1+2\sqrt 3 \sup_{p_o^{-1}]\frac12-\eps, \frac12+\eps[}(q^T-\tilde q^T)\norm{f}_\infty\norm{f'}_\infty
 +\cO(\hbar^{1-2\delta})\right]\left\Vert \frac{\tilde\cP_T-z}\hbar u  \right\Vert
\end{multline*}
which finishes the proof of Proposition \ref{p:plagiat}.

\end{document}